\documentclass[noinfoline]{article}

\RequirePackage[OT1]{fontenc}
\RequirePackage[
amsthm,amsmath
]{imsart}

\usepackage{graphicx}
\usepackage{latexsym,amsmath}
\usepackage{amsmath,amsthm,amscd}
\usepackage{amsfonts}
\usepackage[psamsfonts]{amssymb}
\usepackage{enumerate}

\usepackage{url}

\usepackage{tocvsec2}

\usepackage{epstopdf}
\usepackage{wrapfig}
\usepackage{float}

\usepackage{hyperref}

\startlocaldefs
\numberwithin{equation}{section}

\theoremstyle{plain} 
\newtheorem{theorem}{Theorem}[section]
\newtheorem{corollary}[theorem]{Corollary}

\newtheorem{lemma}[theorem]{Lemma}

\newtheorem{proposition}[theorem]{Proposition}

\theoremstyle{definition} 

\theoremstyle{definition} 

\newtheorem*{ex*}{Example}
\theoremstyle{remark} 

\theoremstyle{remark} 
\newtheorem{remark}[theorem]{Remark}
\newtheorem*{remark*}{Remark}
\numberwithin{equation}{section}
 
\newcommand{\beqa}{\begin{eqnarray}}
\newcommand{\eeqa}{\end{eqnarray}}
 
\newcommand{\bseq}{\begin{subequations}}
 
\newcommand{\eseq}{\end{subequations}}

\newcommand{\dd}{\partial}

\newcommand{\alt}{{\,\operatorname{alt}}}
\newcommand{\eff}{{\,\operatorname{eff}}}

\renewcommand{\dd}{{\,\operatorname{d}}}

\newcommand{\sign}{\operatorname{sign}}
\newcommand{\supp}{\operatorname{supp}}

\newcommand{\vBE}{{\operatorname{vBE}}}

\newcommand{\Ga}{\Gamma}
\newcommand{\si}{\sigma}

\newcommand{\ka}{\kappa}
\newcommand{\la}{\lambda}
\newcommand{\ga}{\gamma}
\newcommand{\de}{\delta}
\newcommand{\be}{\beta}

\newcommand{\ii}[1]{\,\mathbf{I}\{#1\}} 

\newcommand{\pd}[2]{\frac{\partial#1}{\partial#2}} 
\newcommand{\fd}[2]{\frac{\dd#1}{\dd#2}} 

\newcommand{\intr}[2]{\overline{#1,#2}}

\renewcommand{\P}{\operatorname{\mathsf{P}}} 
\newcommand{\E}{\operatorname{\mathsf{E}}}

\newcommand{\R}{\mathbb{R}}
\newcommand{\C}{\mathcal{C}}
\newcommand{\F}{\mathcal{F}}

\newcommand{\G}{\mathcal{G}}

\newcommand{\ta}{{\tilde{a}}}

\newcommand{\tl}{{\tilde{\ell}}}
\newcommand{\tp}{{\tilde{p}}}

\newcommand{\tka}{{\tilde{\kappa}}}
\newcommand{\trho}{{\tilde{\rho}}}

\newcommand{\tR}{{\tilde{R}}}

\newcommand{\tU}{{\tilde{U}}}

\newcommand{\tC}{{\tilde{C}}}

\renewcommand{\le}{\leqslant}
\renewcommand{\ge}{\geqslant}

\newcommand{\X}{{\mathfrak{X}}}
\newcommand{\W}{{\mathfrak{P}}}

\endlocaldefs

\begin{document}

\begin{frontmatter}

\title{Best possible bounds \\ of 
the von Bahr--Esseen type
}
\runtitle{Best bounds of the von Bahr--Esseen type
}

%

\begin{aug}
\author{\fnms{Iosif} \snm{Pinelis}
  }
\runauthor{Iosif Pinelis}


\address{Department of Mathematical Sciences\\
Michigan Technological University\\
Houghton, Michigan 49931, USA\\
E-mail: \printead[ipinelis@mtu.edu]{e1}}
\end{aug}

\begin{abstract}
The well-known von Bahr--Esseen bound on the absolute $p$th moments of martingales with $p\in(1,2]$ is extended to a large class of moment functions, and now with a best possible constant factor (which depends on the moment function). 
This result appears to be new even for the power moments. 
As an application, measure concentration inequalities for separately Lipschitz functions on product spaces are obtained. 
Relations with $p$-uniformly smooth and $q$-uniformly convex normed spaces are discussed. 
\end{abstract}

  
%




\end{frontmatter}

\settocdepth{chapter}

\tableofcontents 

\settocdepth{subsubsection}

\theoremstyle{plain} 
\numberwithin{equation}{section}


\section{Summary and discussion}\label{intro} 

\subsection{Summary}\label{summary} 


Given any sequence $(S_j)_{j=1}^n$ of (real-valued) r.v.'s, let $X_j:=S_j-S_{j-1}$ denote the corresponding differences, for $j\in\intr1n$, with the convention $S_0:=0$, so that $X_1=S_1$; here and in what follows, for any $m$ and $n$ in the set $\{0,1,\dots,\infty\}$ we let $\intr mn$ stand for the set of all integers $i$ such that $m\le i\le n$. 

If $\E|X_j|<\infty$ and $\E(X_j|S_{j-1})=0$ for all $j\in\intr2n$, let us say that the sequence $(S_j)_{j=1}^n$ is a \emph{v-martingale} (where ``v'' stands for ``virtual''); in such a case, let us also say that $(X_j)_{j=1}^n$ is a \emph{v-martingale difference sequence}, or simply that the $X_j$'s are v-martingale differences. 
Note that, for a general v-martingale difference sequence $(X_j)_{j=1}^n$, $X_1$ may be any r.v.\ whatsoever; in particular, its mean (if it exists) may or may not be $0$. 
It is clear that any martingale $(S_j)_{j=1}^n$ is a v-martingale.  
Quite similarly one can define v-martingales with values in a normed space. 

Introduce the following class of generalized moment functions: 
\begin{alignat}{2}
	\F_{1,2}&:=\big\{f\in C^1(\R)\colon & & f(0)=0, \text{ $f$ is even,} \notag\\ 
										&	& &	\text{$f'$ is nondecreasing and concave on $[0,\infty)$}\big\} \notag\\
	&=\big\{f\in C^1(\R)\colon & & f(0)=0, \text{ $f$ is even,} \notag\\ 
										&	&	& \text{$f''$ is nonnegative and nonincreasing on $(0,\infty)$}\big\}; \label{eq:f''}
\end{alignat}
here, as usual, $C^1(\R)$ is the class of all continuously differentiable real-valued functions on $\R$, and then $f''$ denotes the right derivative on $(0,\infty)$ of 
$f'$; on $(-\infty,0)$, $f''$ will denote the left derivative of $f'$. 
It is clear that each function $f\in\F_{1,2}$ is convex and hence nonnegative. Also,  
for each function $f\in\F_{1,2}$ one has $f'(0)=0$. It follows that $f'>0$ on $(0,\infty)$ and hence $f>0$ on $\R\setminus\{0\}$
for any function $f\in\F_{1,2}\setminus\{0\}$.  

\begin{theorem}\label{th:} \ 
\begin{enumerate}[(I)]
	\item 
For any $f\in\F_{1,2}\setminus\{0\}$, $n\in\intr2\infty$, and v-martingale $(S_j)_{j=1}^n$, 
\begin{equation}\label{eq:}
	\E f(S_n)\le\E f(X_1)+C\sum_{j=2}^n\E f(X_j) 
\end{equation}
with $C=C_f$, 
where 
\begin{gather}
	C_f:=\sup_{0<x<s<\infty}\frac{L_{f;s}(x)}{f(s)}, \label{eq:C_f}\\ 
	L_{f;s}(x):=f(x-s)-f(x)+sf'(x). \label{eq:L}   
\end{gather} 
\item 
The constant factor $C_f$ is the best possible in the sense that, 
for each $f\in\F_{1,2}\setminus\{0\}$ and each $n\in\intr2\infty$, the number $C_f$ is the 
smallest value of $C$ such that inequality \eqref{eq:} holds for all v-martingales $(S_j)_{j=1}^n$; in fact, $C_f$ is the best possible even if the differences $X_1,\dots,X_n$ are assumed to be any independent zero-mean r.v.'s.   
\item 
For each $f\in\F_{1,2}\setminus\{0\}$, 
\begin{equation}\label{eq:1<C_f<2}
	1\le C_f\le2. 
\end{equation} 
\item
For each $C\in[1,2]$ there is some $f\in\F_{1,2}\setminus\{0\}$ such that $C_f=C$; in particular, it follows that the bounds $1$ and $2$ on $C_f$ in \eqref{eq:1<C_f<2} are the best possible ones. 
\end{enumerate}
\end{theorem}

Since all functions $f$ in $\F_{1,2}$ are nonnegative, the expressions on both sides of inequality \eqref{eq:} are well defined. 
At that, it is possible for the right-hand side, or for both sides, of \eqref{eq:} to equal $\infty$. 
In the case when the differences $X_1,\dots,X_n$ are independent zero-mean r.v.'s, if the left-hand side of \eqref{eq:} is finite then (by Jensen's inequality) $\E f(X_j)<\infty$ for each $j\in\intr1n$, so that 
the right-hand side is finite as well;   
thus, for independent zero-mean $X_1,\dots,X_n$, 
the two sides of \eqref{eq:} are either both finite or both infinite.  


\subsection{Discussion}\label{discussion} 

In this subsection, we shall 
\begin{enumerate}
	\item describe the structure of the class $\F_{1,2}$ as a convex cone, which will be useful in most of the proofs, and 
provide examples of functions in the class $\F_{1,2}$, including the (absolute) power functions and ``extreme'' functions (that is, functions belonging to the extreme rays of the convex cone $\F_{1,2}$); 
	\item present a general approach to effective calculation of the best possible constant $C_f$, with further information on this constant for the power functions and ``extreme'' functions; 
	\item give an application to the concentration of measure for separately Lipschitz functions on product spaces; 
	\item state other corollaries of the main theorem and relate the results with the relevant ones in the literature, by von Bahr and Esseen (vBE) and other authors. 
\end{enumerate} 
Each of these items will be presented in a separate subsubsection. 

\subsubsection{Structure of the class $\F_{1,2}$ and examples of functions in this class}\label{F_12} 

The following proposition describes the convex-cone structure of the class $\F_{1,2}$. 

\begin{proposition}\label{prop:F12} \ 
\begin{enumerate}[(I)]
	\item 
A function $f\colon\R\to\R$ belongs to the class $\F_{1,2}$ if and only if there exists a (nonnegative, possibly infinite) Borel measure $\ga=\ga_f$ on $(0,\infty]$ such that $\int_{(0,\infty]}(t\wedge1)\ga(\dd t)<\infty$ and 
\begin{equation}\label{eq:f}
	f(x)=\int_{(0,\infty]}\psi_t(x)\ga(\dd t)
\end{equation}
for all $x\in\R$, where 
\begin{equation*}
	\psi_t(x):=x^2-(|x|-t)_+^2,
\end{equation*}
assuming the conventions $u_+:=0\vee u$, $u_+^p:=(u_+)^p$, $u-\infty:=-\infty$, and $(-\infty)_+:=0$, for all real $u$, so that $\psi_\infty(x)=x^2$ for all $x\in\R$. 
Also, 
\begin{equation}\label{eq:psi_0+}
\tfrac1{2t}\,\psi_t(x)\underset{t\downarrow0}\longrightarrow|x| 	
\end{equation}
uniformly in $x\in\R$. 
\item 
For each $f\in\F_{1,2}$, the corresponding measure $\ga=\ga_f$ is unique and determined by the condition that 
\begin{equation}\label{eq:ga}
\ga\big((x,\infty]\big)=\tfrac12\,f''(x) 	
\end{equation}
for all $x\in(0,\infty)$. 
\item 
For any $f\in\F_{1,2}$ and $x\in[0,\infty)$, 
\begin{equation}\label{eq:f'}
	f'(x)=\int_{(0,\infty]}\psi'_t(x)\ga(\dd t)=2\int_{(0,\infty]}(x\wedge t)\ga(\dd t). 
\end{equation}
\end{enumerate}
\end{proposition}

Proposition~\ref{prop:F12} will be used in the proofs of most of the other results of this paper. 

Note that the rays $\R_+\psi_t$ corresponding to
the functions $\psi_t$ \big(for $t\in(0,\infty]$\big) are precisely the extreme rays of the convex cone $\F_{1,2}$, where $\R_+f:=\{\la f\colon\la\in(0,\infty)\}$, for any $f\in\F_{1,2}\setminus\{0\}$. This follows because the rays $\R_+\ga_{\psi_t}=\R_+\de_t$ \big(with $t\in(0,\infty]$\big) 
are precisely the extreme rays of the corresponding convex cone $\{\ga_f\colon f\in\F_{1,2}\}$ of measures, where $\de_t$ stands for the Dirac measure at the point $t$. \big(A ray $\R_+f$ of a convex cone is called extreme if, for any nonzero $f_1$ and $f_2$ in the cone such that $f_1+f_2=f$, both $f_1$ and $f_2$ must lie on the ray.\big) 

Also, note that $\psi_t(x)=x^2\ii{|x|<t}+(2t|x|-t^2)\ii{|x|\ge t}$, so that $\psi_t(x)$ equals $x^2$ for small enough $|x|$ and is asymptotic to $2t|x|$ as $|x|\to\infty$. Thus, the ``extreme'' function $\psi_t$ is in a sense intermediate between the absolute powers $|\cdot|$ and $|\cdot|^2$. So, by \eqref{eq:f}, all functions $f\in\F_{1,2}$ inherit such a property. This should explain the choice of the notation $\F_{1,2}$. 

Classes of moment functions similar to $\F_{1,2}$ arise naturally in extremal problems in probability and statistics; see e.g.\ \cite{eaton1,utev-extr,T2,zinn_etal,pin98,pin99,spher,bent-liet02,bent-jtp,bent-ap,asymm,normal,pin-hoeff}; $\F_{1,2}$ is especially similar to the class $\mathcal{O}_{2,3}$ considered in \cite{zinn_etal}.  

Let us now give some examples of functions $f$ in $\F_{1,2}$. The ``extreme'' functions $\psi_t$ have been already mentioned. 
Perhaps the most important members of the class $\F_{1,2}$ are the power functions $|\cdot|^p$ with $p\in(1,2]$. The function $|\cdot|$ is not in $\F_{1,2}$, since it is not in $C^1(\R)$. 

It is easy to construct many other kinds of examples of functions $f\in\F_{1,2}$ by (i) letting $f''$ be \big(on $(0,\infty)$\big) any function, say $g$, which is nonnegative, nonincreasing, right-continuous, and integrable on any interval of the form $(0,u]$, for any $u\in(0,\infty)$; then (ii) finding $f$ on $[0,\infty)$ as the solution to the following initial value problem: $f(0)=f'(0)=0$ and $f''=g$ on $(0,\infty)$; and finally (iii) extending $f$ to the entire real line $\R$ as an even function. 

E.g., taking 
$g(x)=(1+x)^{p-2}$ for $p\in(1,2)$ and $x\in(0,\infty)$, one ends up with $f(x)=\frac1{p(p-1)}\,[(1+|x|)^p-1-p|x|]$ for all $x\in\R$, which is asymptotic to $\frac12\,x^2$ as $x\to0$ and to $\frac1{p(p-1)}\,|x|^p$ as $|x|\to\infty$; if the condition $p\in(1,2)$ is replaced here by $p\in(-\infty,0)\cup(0,1)$, then $f(x)$ is asymptotic to  $\frac{|x|}{1-p}$ as $|x|\to\infty$.   
%
Similarly one can get $f(x)\equiv e^{-|x|}-1+|x|$ \big(by starting with $g(x)=e^{-x}$ for $x\in(0,\infty)$\big); $f(x)\equiv|x|-\ln(1+|x|)$ \big(with $g(x)\equiv\frac1{(1+x)^2}$\big); $f(x)\equiv|x|\ln(1+|x|)$ \big(with $g(x)\equiv\frac1{1+x}+\frac1{(1+x)^2}$\big). 

Perhaps a more interesting example is the following family of functions, which are parabolic splines (and will also be used in Remark~\ref{rem:no mono in s}):  
\begin{equation}\label{eq:f_alt}
f_\alt(x):=
   \frac{(|x|-x_j)^2}{2 (x_j+1)^{2/3}}
   + \sum _{k=0}^{j-1}\frac{\left[|x|-\frac{1}{2}(x_k+x_{k+1})\right] (x_{k+1}-x_k) }{(x_k+1)^{2/3}}
\end{equation}
if $x_j\le|x|<x_{j+1}$ and $j\in\intr0\infty$, where $x_0:=0$, $x_1$ is any positive real number, 
and $x_j:=q^{2^{j-1}}-1$ for $q:=x_1+1$ and all $j\in\intr2\infty$, so that $x_{j+1}+1=(x_j+1)^2$ for all $j=1,2,\dots$ (we use the standard conventions $a^{b^c}:=a^{(b^c)}$ and $\sum _{k=0}^{-1}\ldots:=0$). 

It is easy to check that $f_\alt\in\F_{1,2}$ and $f''_\alt(x)=(x_j+1)^{-2/3}=(x_{j+1}+1)^{-1/3}$ if $x_j\le|x|<x_{j+1}$ and $j\in\intr0\infty$, so that 
the function $f''_\alt$ alternates between the powers $(|\cdot|+1)^{-2/3}$ and $(|\cdot|+1)^{-1/3}$, as shown in the left panel of Figure~\ref{fig:oscil}. 
So, one might expect that the function $f_\alt$ alternates (far away from $0$) between something like the powers $|\cdot|^{-2/3+2}=|\cdot|^{4/3}$ and $|\cdot|^{-1/3+2}=|\cdot|^{5/3}$. This expectation is only partially justified. 

Indeed, introduce the (instantaneous) ``effective'' exponent of the function $f_\alt$ at a point $x\in\R\setminus\{0\}$ by the formula  
\begin{equation*}
	p_\eff(x):=\log_{|x|}f_\alt(x),\quad\text{so that}\quad f_\alt(x)=|x|^{p_\eff(x)}. 
\end{equation*}
The following proposition shows that the effective exponent $p_\eff$ eventually, ``in the limit'', alternates between $\frac32$ (rather than the expected $\frac43$) and $\frac53$. In this sense, one might say that $f''_\alt$ stays closer to $(|\cdot|+1)^{-1/3}$ than to $(|\cdot|+1)^{-2/3}$, ``most of the time''. 

\begin{proposition}\label{prop:osc} \ 
\begin{enumerate}[(i)]
	\item $p_\eff(x)=\tp_\eff\big(\rho(x)\big)+o(1)$ as $x\to\infty$, where 
	$\tp_\eff(r):=(2-\frac2{3r})\vee(1+\frac2{3r})$ and $\rho(x):=2^{1-j}\log_q(x+1)$ for $x\in(x_j,x_{j+1}]$.  
	\item For each $j\in\intr1\infty$, the function $\rho$ increases from $1$ to $2$ on the interval $(x_j,x_{j+1}]$. 
\item For each $j\in\intr1\infty$, the approximate effective exponent $\tp_\eff\big(\rho(x)\big)$ decreases from $\frac53$ to $\frac32$ and then increases back to $\frac53$
as $x+1$ increases from $x_j+1$ to $(x_j+1)^{4/3}$ and then on to $x_{j+1}+1=(x_j+1)^2$, respectively. 
\end{enumerate}
\end{proposition}

Part of the graph of the (exact) effective exponent $p_\eff$ (with $x_1=\frac1{10}$) is shown in the right panel of Figure~\ref{fig:oscil}. Recall that the $x_j$'s grow very fast in $j$ for large $j$. Therefore,  
for better presentation, the horizontal axis in the right panel is nonlinearly rescaled so that the $x_j$'s appear equally spaced. Namely, what is actually shown here is part of the graph $\{\big(\log_2\log_q(x+1),p_\eff(x)\big)\colon x>x_1\}$; note that $\log_2\log_q(x_j+1)=j-1$ for all $j\in\intr1\infty$.  

\begin{figure}[!Hhtb]
	\centering		\includegraphics[width=1.00\textwidth]{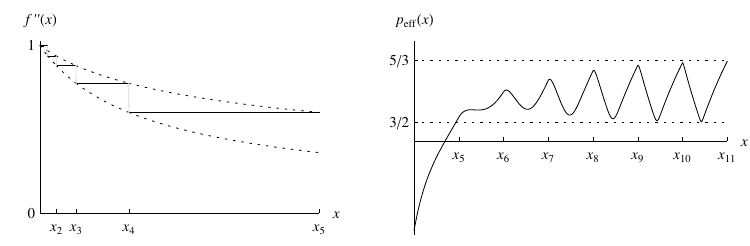}
	\caption{Left panel: $f''$ (solid) for $f=f_\alt$ alternates between $(|\cdot|+1)^{-2/3}$ (dotted) and $(|\cdot|+1)^{-1/3}$ (dotted). 
	Right panel: the effective exponent $p_\eff$ (solid) for $f=f_\alt$ eventually alternates between $\frac32$ (dotted) and $\frac53$ (dotted). }
	\label{fig:oscil}
\end{figure}

\subsubsection{On the best possible constant $C_f$ in general and, in particular, for the power and extreme functions}\label{C_f} 

The following proposition concerns some general properties of the constant factor $C_f$ for nonzero $f$ in $\F_{1,2}$ except for $f=\psi_\infty$; in the latter, trivial case, one has $C_f=1$, as also stated in Proposition~\ref{prop:C_psi}; recall that $\psi_\infty(x)=x^2$ for all $x\in\R$.  

\begin{proposition}\label{prop:C_f}\ 
Take any $f\in\F_{1,2}\setminus\{0,\psi_\infty\}$. Let $s_f:=\inf\supp\ga$, where $\supp\ga$ stands for the support of the measure $\ga=\ga_f$ defined in Proposition~\ref{prop:F12}. 
Recall the definition of $L_{f;s}(x)$ in  \eqref{eq:L}. 
Then the following statements hold. 
\begin{enumerate}[(i)]
\item $s_f\in[0,\infty)$. 
	\item For any $s\in(0,s_f]$, one has $L_{f;s}(x)=f(s)$ for all $x\in(0,s)$. 
	\item For any $s\in(s_f,\infty)$, one has $L'_{f;s}(0+)>0$ and $L'_{f;s}(s-)<0$. 
	\item For any $s\in(0,\infty)$, there is some (not necessarily unique) $x_{f;s}\in(0,s)$ such that 
	$L_{f;s}(x)$ is nondecreasing in $x\in(0,x_{f;s}]$ and nonincreasing in $x\in[x_{f;s},s)$. 
		\item One has 	
\begin{align*}
C_f&=\sup\nolimits_{s\in(s_f,\infty)}\big[\tfrac1{f(s)}\,\max\nolimits_{x\in(0,s)}L_{f;s}(x)\big] 
\\
	&	=\sup\nolimits_{s\in(s_f,\infty)}\big[\tfrac1{f(s)}\,L_{f;s}(x_{f;s})\big]>1. 	
\end{align*}
\end{enumerate}
\end{proposition} 

\begin{remark}\label{rem:no mono in s}
Proposition~\ref{prop:C_f} provides for an effective maximization of $L_{f;s}(x)$ in $x\in(0,s)$, for any given $s\in(0,\infty)$, so that 
$\mathcal{L}_f(s):=\tfrac1{f(s)}\,\max\nolimits_{x\in(0,s)}L_{f;s}(x)=\tfrac1{f(s)}\,L_{f;s}(x_{f;s})$ can be effectively found. In the important special case when $f$ is a power function $|\cdot|^p$ \big(with $p\in(1,2]$\big), one can also use the homogeneity of $f$ in order to compute the constant $C_f$ quite effectively, as described in Proposition~\ref{prop:C_powers}. However, in general it remains to maximize  $\mathcal{L}_f(s)$ in $s\in(s_f,\infty)$. It appears that usually $\mathcal{L}_f(s)$ is monotonically nondecreasing in $s$, if the function $f$ is not too irregular; one ``exceptional'' function $f$ for which $\mathcal{L}_f$ lacks such a monotonicity property is a function $f_\alt$ of the ``alternating'' family described by formula \eqref{eq:f_alt}.   
Indeed, take $f=f_\alt$ with $x_1=\frac15$. 
Then 
$\mathcal{L}(\frac{107}{100})<\mathcal{L}(\frac{106}{100})$. 
One may still ask whether it is true for all $f\in\F_{1,2}$ that the limit $\mathcal{L}_f(\infty-)$ exists, and if so, whether it is true that $\mathcal{L}_f(s)\le\mathcal{L}_f(\infty-)$ for all $s\in(s_f,\infty)$, so that $C_f$ be found as $\mathcal{L}_f(\infty-)$. 
In any case, Theorem~\ref{th:} reduces the problem of finding the optimal constant $C$ in \eqref{eq:} to a maximization just in two real variables, $s$ and $x$, which should not usually be too difficult. 
\end {remark}

Now let us provide a simple description of the constant $C_f$ in the case when $f$ is an ``extreme'' function $\psi_t$, representing the extreme rays of the convex cone $\F_{1,2}$:  

\begin{proposition}\label{prop:C_psi}\ 
One has  
$C_{\psi_t}=2$ for each $t\in(0,\infty)$, whereas $C_{\psi_\infty}=1$. 
\end{proposition} 

\begin{remark}\label{rem:nonlin}
Proposition~\ref{prop:C_psi} might seem quite surprising: whereas, by Theorem~\ref{th:}, the range of the values of $C_f$ over all nonzero $f$ in the convex cone $\F_{1,2}$ is the entire interval $[1,2]$, the only value that $C_f$ takes on all the extreme rays $\R_+\psi_t$ \big(which span the cone $\F_{1,2}$ in the sense of \eqref{eq:f}\big) is $2$. This suggests strong nonlinearity of the optimal constant factor $C_f$ in $f$. 
However, as seen from the proof of Proposition~\ref{prop:C_psi}, the fact that $C_{\psi_t}$ is the same for all $t\in(0,\infty)$ is due to a simple homogeneity property. Note also the discontinuity of $C_{\psi_t}$ in $t$ at $t=\infty$.
\end{remark}

As mentioned earlier, for any $p\in(1,2]$ the power function $|\cdot|^p$ belongs to the class $\F_{1,2}$; for such $p$, consider the corresponding constant factor 
\begin{equation*}
	\tC_p:=C_{|\cdot|^p}, 
\end{equation*} 
so that for any v-martingale $(S_j)_{j=1}^n$ 
\begin{equation}\label{eq:powers}
	\E|S_n|^p\le\E|X_1|^p+\tC_p\sum_{j=2}^n\E|X_i|^p.  
\end{equation}
Note that $|\cdot|^2=\psi_\infty$, so that, by Proposition~\ref{prop:C_psi}, 
\begin{equation}\label{eq:tC_2}
	\tC_2=1. 
\end{equation}

\begin{proposition}\label{prop:C_powers}\ 
\begin{enumerate}[(i)]
\item For any $p\in(1,2)$  
\begin{equation*}
	\tC_p=\ell(p,x_p)=\max_{x\in(0,1)}\ell(p,x), 
\end{equation*}
where 
\begin{equation}\label{eq:ell}
	\ell(p,x):=L_{|\cdot|^p;1}(x)=(1-x)^p-x^p+px^{p-1}
\end{equation}
for $x\in(0,1)$, and $x_p$ is the only root $x\in(0,1)$ of the equation 
\begin{equation}\label{eq:x_p eq}
	(1-x)^{p-1}+x^{p-1}=(p-1)x^{p-2}. 
\end{equation}
Moreover, $\ell(p,x)$ is increasing in $x\in(0,x_p)$ and decreasing in $x\in(x_p,1)$, for each $p\in(1,2)$. 
\item In fact, $x_p\in(\frac{p-1}5,\frac{p-1}2)\subset(0,\frac12)$ for all $p\in(1,2)$. 
\item Further, $\tC_p$ is continuously (and strictly) decreasing in $p\in(1,2]$ from $\tC_{1+}=2$ to $\tC_2=1$; furthermore, $\tC_p$ is real-analytic in $p\in(1,2)$. 
\item The values $\tC_p$ are algebraic for all rational $p\in(1,2]$; in particular, $\tC_{3/2}=\sqrt{1+\frac{1}{\sqrt{2}}}=1.306\dots$ (with $x_{3/2}=\frac{1}{4} \left(2-\sqrt{2}\right)=0.146\dots$). 
\item Explicit upper and lower bounds on $\tC_p$ are given by the inequalities
\begin{equation}\label{eq:<tC_p<}
\tC_p^{-,1}\vee\tC_p^{-,2}<\tC_p<\tC_p^{+,1}\wedge\tC_p^{+,2}\le\tC_p^{+,2}<W_p	
\end{equation}
for all $p\in(1,2)$, where 
\begin{align*}
	\tC_p^{-,1}:=&2^{-p} \big((3 - p)^p + (p-1)^{p - 1} (p + 1)\big), \\
	\tC_p^{-,2}:=&5^{-p} \big((6 - p)^p + (p-1)^{p - 1} (4 p+1)\big), \\
	\tC_p^{+,1}:=&\tfrac{2^{- p}}{50(3 - p)}  \big((p-1)^{p-1} (150 + 181 p - 152 p^2 + 21 p^3) \\
	&\qquad\quad+ (3 - p)^{p-1}(450 - 381 p + 152 p^2 - 21 p^3)\big), \\
	\tC_p^{+,2}:=&\tfrac{5^{-p}}{8(6 - p)}  \big(4(p-1)^{p-1} (12 - 35 p + 94 p^2 - 21 p^3) \\
	&\qquad\quad+ (6 - p)^{p-1}(288 - 15 p - 94 p^2 + 21 p^3)\big), \\
	W_p:=&2^{2-p}. 
\end{align*}
The upper bound $W_p$ on $\tC_p$ is exact at the endpoints of the interval $(1,2)$ in the sense that $\tC_{1+}=W_{1+}$ and $\tC_2=\tC_{2-}=W_{2-}=W_2$;  
each of the bounds $\tC_p^{-,1}$, $\tC_p^{-,2}$, $\tC_p^{+,1}$, and $\tC_p^{+,2}$ is also exact in the similar sense.  
\end{enumerate}
\end{proposition}

\begin{wrapfigure}{l}{.45\textwidth}
  \vspace{-12pt}
    \includegraphics[width=.45\textwidth]{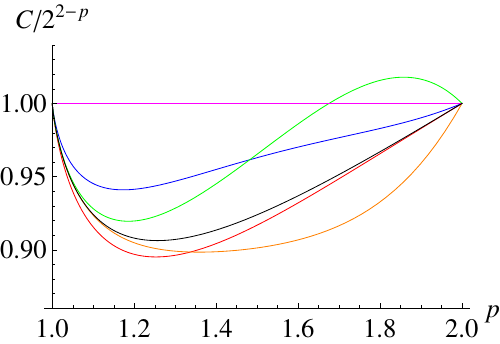}
  \caption{The ratios of $\tC_p$ (black), $\tC_p^{-,1}$ (red), $\tC_p^{-,2}$ (orange), $\tC_p^{+,1}$ (green), $\tC_p^{+,2}$ (blue), and $W_p$ (magenta) to $2^{2-p}$.}
	\label{fig:tC_p,upper-lower}
	    \vspace{-22pt}
\end{wrapfigure}

The graphs of the ratios of $\tC_p$, $\tC_p^{-,1}$, $\tC_p^{-,2}$, $\tC_p^{+,1}$, $\tC_p^{+,2}$, and $W_p$ to $W_p=2^{2-p}$ are shown in Figure~\ref{fig:tC_p,upper-lower}. 
The graph of $\tC_p$, in comparison with $W_p$ and the von Bahr--Esseen constant $C_p^\vBE$, is presented in Figure~\ref{fig:tC_p}.


As mentioned in Subsubsection~\ref{F_12}, the absolute-value function $|\cdot|$ is is not in the class $\F_{1,2}$. However, by \eqref{eq:psi_0+}, $|\cdot|$ is in the closure of $\F_{1,2}$ with respect to the uniform convergence on $\R$. 
It is also clear that inequality \eqref{eq:} holds for $f=|\cdot|$ 
(and any r.v.'s $X_1,\dots,X_n$) with $C=\tC_1:=1$. From this viewpoint, there is a discontinuity of $C_p$ at $p=1$, namely, $\tC_{1+}=2\ne1=\tC_1$. 

\subsubsection{Application: concentration inequalities for separately Lipschitz functions on product spaces}\label{concentr} 




Let $X_1,\dots,X_n$ be independent r.v.'s with values in measurable spaces $\X_1,\dots,\X_n$, respectively. Let $g\colon\W\to\R$ be a measurable function on the product space $\W:=\X_1\times\dots\times\X_n$. 
Let us say (cf.\ \cite{bent-isr,normal}) that $g$ is {\em separately Lipschitz} if it satisfies a Lipschitz type condition in each of its arguments:
\begin{equation}\label{eq:Lip}
|g(x_1,\dots,x_{i-1},\tilde x_i,x_{i+1},\dots,x_n) -
g(x_1,\dots,x_n)| \le \rho_i(\tilde x_i,x_i)
\end{equation}
for some measurable functions $\rho_i\colon\X_i\times\X_i\to\R$ and 
all $i\in\intr1n$, $(x_1,\dots,x_n)\in\W$, and $\tilde x_i\in\X_i$. 

Take now any separately Lipschitz function $g$ and let  
$$Y:=g(X_1,\dots,X_n).$$
Suppose that the r.v.\ $Y$ has a finite mean. 
Then one has the following. 

\begin{corollary}\label{cor:concentr}
For each $i\in\intr1n$, take any $x_i\in\X_i$. 
\begin{enumerate}[(I)]
	\item 
For any $f\in\F_{1,2}\setminus\{0\}$
\begin{equation}\label{eq:concentr}
	\E f(Y)\le f(\E Y)+\ka_f C_f\sum_{i=1}^n\E f\big(\rho_i(X_i,x_i)\big),    
\end{equation}
where 
\begin{gather}
	\ka_f:=\sup\Big\{\frac{U_f(c,s,0)}{U_f(c,s,a)}\colon s\in(0,\infty),\ c\in(0,\tfrac s2),\ a\in(0,c)\Big\} \in[1,2], \label{eq:ka_f} 
	\\
	U_f(c,s,a):=cf(s-c+a)+(s-c)f(a-c) \label{eq:U_f}  
\end{gather}
\big(the above definition of $\ka_f$ is valid, because $f>0$ on $\R\setminus\{0\}$ and hence $U_f(c,s,a)>0$ for any 
$s\in(0,\infty)$, $c\in(0,\tfrac s2)$, and $a\in(0,c)$\big). 
	\item For any $p\in(1,2]$
\begin{equation}\label{eq:concentr-p}
	\E|Y|^p\le|\E Y|^p+\tka_p\tC_p\sum_{i=1}^n\E\big|\rho_i(X_i,x_i)\big|^p,   
\end{equation}
where 
\begin{equation}\label{eq:tka_p}
	\tka_p:=\ka_{|\cdot|^p}=\max_{c\in[0,1/2]}\big[(c^{p - 1} + (1 - c)^{p - 1}) \big(c^{\frac1{p - 1}} + (1 - c)^{\frac1{p - 1}}\big)^{p - 1}\big]. 
\end{equation}
Moreover, $\tka_p$ continuously and strictly decreases in $p\in(1,2]$ from $2$ to $1$. 
Furthermore, the values of $\tka_p$ are algebraic for all rational $p\in(1,2]$; in particular,  
$\tka_{3/2}=\frac19\,\sqrt{51 + 21\sqrt7}=1.14\dots$, corresponding to $c=\frac16\,(3 -\sqrt{1 + 2\sqrt7})=0.081\dots$ in \eqref{eq:tka_p}. 
The graph of $\tka_p$ is shown 
below. 
\end{enumerate}
\end{corollary} 

\vspace*{-0pt}

\parbox{.45\textwidth}
{
    \includegraphics[width=.4\textwidth]{
		concentr-powers.
		pdf} \\ 
\ \centering {\small $\tka_p$, solid; \quad $1$, dotted.}
}
\hspace*{.02\textwidth}
\parbox{.48\textwidth}
{
One can observe some similarity between $C_f,\tC_p$ and $\ka_f,\tka_p$. 
Thus, going from the ``one-dimensional'' inequality \eqref{eq:} or \eqref{eq:powers} for v-martingales to the ``multi-dimensional'' measure concentration inequality \eqref{eq:concentr} or \eqref{eq:concentr-p} entails an extra factor, $\ka_f$ or $\tka_p$, whose values are between $1$ and $2$. 
}


The proof of Corollary~\ref{cor:concentr} is partly based on the following proposition, which may be of independent interest. \\ 

\begin{proposition}\label{lem:centring}
For any zero-mean r.v.\ $X$, $f\in\F_{1,2}\setminus\{0\}$, and $a\in\R$
\begin{equation}\label{eq:centring}
	\E f(X)\le\ka\E f(X+a)
\end{equation}
with $\ka=\ka_f$, and $\ka_f$ is the best possible constant $\ka$ in \eqref{eq:centring}. 
\end{proposition}

In turn, the proof of Proposition~\ref{lem:centring} uses 

\begin{proposition}\label{prop:minU}
Take any $f\in\F_{1,2}\setminus\{0\}$, $s\in(0,\infty)$, and $c\in(0,\tfrac s2)$. Then $U_f(c,s,a)$ (defined in \eqref{eq:U_f}) is convex in $a\in\R$. Moreover, $U_f(c,s,a)$ attains its minimum over all $a\in\R$ at a unique point $a_{f;c,s}\in[0,c)$. In particular, for all $t\in(0,\infty)$, $s\in(0,\infty)$, and $c\in(0,\tfrac s2)$
\begin{equation}\label{eq:a_t,c,s}
	a_{\psi_t;c,s}=\tfrac c{s-c}\,(s-c-t)_+
\end{equation}
and $\ka_{\psi_t}=2$. 
\end{proposition} 

On the other hand, Proposition~\ref{prop:minU} obviously complements Corollary~\ref{cor:concentr}. 

A difficulty in proving the uniqueness of the minimizer of $U_f(c,s,a)$ in $a$ in 
Proposition~\ref{prop:minU} is that, in general, $U_f(c,s,a)$ is not strictly convex in $a$.  


An example of separately Lipschitz functions $g:\X^n\to\R$ is given by the formula 
$g(x_1,\dots,x_n)=\|x_1+\dots+x_n\|$ 
for all $x_1,\dots,x_n$ in a separable Banach space $(\X,\|\cdot\|)$. 
In this case, one may take $\rho_i(\tilde x_i,x_i)\equiv\|\tilde x_i-x_i\|$. 
Thus, one obtains 

\begin{corollary}\label{cor:conc-sums}
Let $X_1,\dots,X_n$ be independent random vectors in the Banach space $(\X,\|\cdot\|)$. 
Let $S_n:=X_1+\dots+X_n$. 
For each $i\in\intr1n$, take any $x_i\in\X$. 
Then 
for any $f\in\F_{1,2}\setminus\{0\}$
\begin{equation}\label{eq:conc-sums}
	\E f(\|S_n\|)\le f(\E\|S_n\|)+\ka_f C_f\sum_{i=1}^n\E f\big(\|X_i-x_i\|\big).    
\end{equation}
Moreover, for any $p\in(1,2]$
\begin{equation}\label{eq:conc-p-sums}
	\E\|S_n\|^p\le(\E\|S_n\|)^p+\ka_p\tC_p\sum_{i=1}^n\E\|X_i-x_i\|^p.  
\end{equation}
\end{corollary}

For $p=2$, inequality \eqref{eq:conc-p-sums} was obtained in \cite[Theorem~4]{pin-sakh}, based on an improvement the method of Yurinski\u\i (1974) \cite{yurinskii}; cf.\ \cite{mcdiarmid89,mcdiarmid98,bent-isr}, \cite[Section~4]{normal}, and \cite[Proposition~2.5]{pin94}. The proof of Corollary~\ref{cor:concentr} 
is based in part on the same kind of improvement.

As can be seen from that proof, both Corollaries~\ref{cor:concentr} and \ref{cor:conc-sums} will hold even if 
the separately-Lipschitz condition \eqref{eq:Lip} 
is relaxed to 
\begin{equation}\label{eq:LipE}
	|\E g(x_1,\dots,x_{i-1},\tilde x_i,X_{i+1},\dots,X_n) 
- 
\E g(x_1,\dots,x_i,X_{i+1},\dots,X_n)|\le \rho_i(\tilde x_i,x_i). 
\end{equation}

Note also that in Corollaries~\ref{cor:concentr} and \ref{cor:conc-sums} the r.v.'s $X_i$ do not have to be zero-mean, or even to have any definable mean; at that, the arbitrarily chosen $x_i$'s may act as the centers, in some sense, of the distributions of the corresponding $X_i$'s.  

Clearly, the separate-Lipschitz (sep-Lip) condition \eqref{eq:Lip} is easier to check than a joint-Lipschitz one. Also, sep-Lip (especially in the relaxed form \eqref{eq:LipE}) is more generally applicable. On the other hand, when a joint-Lipschitz condition is satisfied, one can generally obtain better bounds. Literature on the concentration of measure phenomenon, almost all of it for joint-Lipschitz settings, is vast; let us mention here only \cite{ledoux-tala,ledoux_book,lat-olesz,bouch-etal,ledoux-olesz}.  

\subsubsection{Other corollaries of Theorem~\ref{th:} and comparisons with known results}\label{cor's,relations} 
Take any $p\in(1,2]$. 
A normed space $(\X,\|\cdot\|)$ (or, briefly, $\X$) is called $p$-uniformly smooth \cite{BCL} if for some constant $D\in(0,\infty)$ (referred to as a $p$-uniform smoothness constant of $\X$) and all $x$ and $y$ in $\X$ one has $\tfrac12\,(\|x+y\|^p+\|x-y\|^p)\le\|x\|^p+D^p\|y\|^p$ or, equivalently, 
\begin{equation}\label{eq:p-smooth}
	\E\|x+X y\|^p\le \|x\|^p+D^p\E|X|^p\|y\|^p 
\end{equation}
for all symmetric(ally distributed) real-valued r.v.\ $X$. 
If $\X$ is $p$-uniformly smooth with a $p$-uniform smoothness constant $D$, 
let us say that $\X$ is $(p,D)$-uniformly smooth or, simply, $(p,D)$-smooth. 
For instance, for any $q\in[2,\infty)$ the space $L^q(\mu)$ is $(2,D)$-smooth with $D=\sqrt{q-1}$, which is the best possible constant of the $2$-uniform smoothness as long as the space $L^q(\mu)$ is at least two-dimensional --- see \cite[Proposition~2.1]{pin94}, \cite[Proposition~3]{BCL}, \cite[Corollary~2.8]{wellner_nemir}. 

Dual to the notion of $(p,D)$-uniform smoothness is that of $(q,D^{-1})$-uniform convexity, whose definition can be obtained by reversing the inequality sign in \eqref{eq:p-smooth} and replacing there $p$ and $D$ by $q$ and $D^{-1}$, respectively; here, $\frac1p+\frac1q=1$. 
In particular, a result due to Ball, Carlen, and Lieb 
\cite[Lemma~5]{BCL} is that $\X$ is $(p,D)$-uniformly smooth iff its dual $\X^*$ is $(q,D^{-1})$-uniformly convex; cf.\ e.g.\  \cite{day44,lindenstr63}. 
Note that $q$-uniform convexity and $p$-uniform smoothness are refinements of
the notions of uniform convexity and uniform smoothness, which go back to Clarkson \cite{clarkson} and Day \cite{day44}; cf.\ \cite{hoff-jorg,woycz75}. These notions  
are important in functional analysis. 
In particular, Pisier \cite{pisier75} showed that every super-reflexive space is $q$-uniformly convex and $p$-uniformly smooth for some $q$ and some $p$; an earlier result due to Enflo \cite{enflo} stated that $\X$ is super-reflexive iff it 
is isomorphic to a uniformly convex space. 
Among many other results, Pisier \cite{pisier75} also showed that the super-reflexivity is equivalent to the super-Radon-Nikodym property. 
Applications of the 2-uniform convexity/2-uniform smoothness to Finsler manifolds were given by 
Ohta \cite{ohta}. 

It is clear that $\X$ is $(p,D)$-smooth iff inequality \eqref{eq:} with $C=D^p$ and $f=\|\cdot\|^p$ holds for all martingales (or even v-martingales) $(S_j)_{j=1}^n$ with values in $\X$ and conditionally symmetric differences $X_2,\dots,X_n$; by symmetrization, the same inequality will then hold without the conditional symmetry restriction, but with the worse constant $C=(2D)^p$ instead of $C=D^p$. 
These considerations suggest the following. 

Let us say that the space $\X$ is \emph{completely $(p,D)$-smooth} if inequality \eqref{eq:p-smooth} holds for all \emph{zero-mean} real-valued r.v.'s $X$ (and all $x$ and $y$ in $\X$). It is clear that $\X$ is completely $(p,D)$-smooth iff inequality \eqref{eq:} with $C=D^p$ and $f=\|\cdot\|^p$ holds for all martingales (or even v-martingales) $(S_j)_{j=1}^n$ with values in $\X$. Also, Proposition~\ref{prop:C_powers} immediately implies 

\begin{corollary}\label{cor:L^p} 
Take any $p\in(1,2]$ and any measure $\mu$ on any measurable space.  
Then the space $L^p(\mu)$ is completely $(p,D)$-smooth with the best possible constant $D=\tC_p^{1/p}$. 
So, 
for any $n\in\intr2\infty$ and v-martingale $(S_j)_{j=1}^n$ with values in $L^p(\mu)$, 
\begin{equation}\label{eq:L^p}
	\E\|S_n\|_p^p\le\E\|X_1\|_p^p+\tC_p\sum_{j=2}^n\E\|X_j\|_p^p  
\end{equation} 
(cf.\ \eqref{eq:conc-p-sums}). 
\end{corollary}


The above discussion suggests that the form of inequality \eqref{eq:} is rather natural in such contexts as concentration of measure, uniform smoothness, and martingales (or v-martingales). Yet, in the case when the differences $X_1,\dots,X_n$ are independent real-valued zero-mean r.v.'s, the form of 
the following immediate corollary of Theorem~\ref{th:} may be more relevant. 

\begin{corollary}\label{cor:}
For any $f\in\F_{1,2}\setminus\{0\}$, $n\in\intr2\infty$, and (real-valued) v-martin\-gale $(S_j)_{j=1}^n$, 
\begin{equation}\label{eq:spread}
	\E f(S_n)\le K\sum_{j=1}^n\E f(X_j) 
\end{equation}
with $K=C_f$. 
\end{corollary}

However, in inequality \eqref{eq:spread} the constant factor $K=C_f$ is no longer the best possible one, at least for independent zero-mean $X_j$'s. 
One way to reduce the constant is as follows. 
In the conditions of Corollary~\ref{cor:}, rewrite the right-hand side of \eqref{eq:} with $C=C_f$ as $C_f\sum_{j=1}^n\E f(X_j)-(C_f-1)\E f(X_1)$. Then, assuming that 
$\E f(X_1)\ge\frac\la n\,\sum_{j=1}^n\E f(X_j)$ for some $\la\in(0,\infty)$, one sees that the constant factor $K=C_f$ in \eqref{eq:spread} can  be reduced by spreading the ``excess'' $C_f-1\ge0$ over all the summands $\E f(X_1),\dots,\E f(X_n)$, to get \eqref{eq:spread} with  
\begin{equation}\label{eq:M}
	K=C_f-\tfrac\la n\,(C_f-1)\le C_f. 
\end{equation}

To develop this simple observation a bit further, let us take any $\la\in(0,\infty)$ and say that 
a sequence $(S_j)_{j=1}^n$ is 
a \emph{$\la$-good rearranged-v-martingale} if there are 
(i) some $i\in\intr1n$ such that 
$\E f(X_i)\ge\frac\la n\,\sum_{j=1}^n\E f(X_j)$ and (ii)   
a permutation $(j_1,\dots,j_{n-1})$ of the set $\intr1n\setminus\{i\}$ such that  $(X_i,X_{j_1},\dots,X_{j_{n-1}})$ is the difference sequence of a v-martingale.
Note that, if the differences $X_1,\dots,X_n$ of a sequence $(S_j)_{j=1}^n$ are independent zero-mean r.v.'s, then $(S_j)_{j=1}^n$ is a $1$-good rearranged-v-martingale. 
(In general, a $\la$-good rearranged-v-martingale does not have to be a v-martingale.) 
Thus, one obtains  

\begin{corollary}\label{cor:1n}
For any $f\in\F_{1,2}\setminus\{0\}$, $n\in\intr2\infty$, and $\la$-good rearranged-v-martingale $(S_j)_{j=1}^n$, inequality \eqref{eq:spread} holds, again with $K$ as in \eqref{eq:M}. 
\end{corollary}

In the special case of the power functions $|\cdot|^p$ \big(with $p\in(1,2)$\big) in place of general $f\in\F_{1,2}\setminus\{0\}$, an inequality of the form \eqref{eq:spread} was obtained by von Bahr and Esseen (vBE) \cite{bahr65}: 
\begin{equation}\label{eq:vBE}
	\E|S_n|^p\le K\sum_{j=1}^n\E|X_j|^p,    
\end{equation}
with the constant factor $K=2-\frac1n=2-\frac1n(2-1)$, which, by part (iii) of Proposition~\ref{prop:C_powers}, is greater than the $K$ in \eqref{eq:M}, again for $f=|\cdot|^p$ with $p\in(1,2)$.  
The vBE inequality \eqref{eq:vBE} has been used in various kinds of studies, see e.g.\  \cite{bel-smir,barral,surg,mor-weiss,manst07,gonz,molina,ben-arous,cohen-lin,johnson,fan-ai,barral99,put-vZwet}, among the more recent articles. 

As noted by vBE \cite{bahr65}, the
special case of inequality \eqref{eq:vBE} (with $K=1$) 
when the conditional distributions of the differences $X_i$ given $S_{i-1}$ are symmetric for all $i\in\intr2n$  easily follows from Clarkson's inequality \cite{clarkson} 
\begin{equation}\label{eq:clarkson}
	|x+y|^p+|x-y|^p\le2|x|^p+2|y|^p
\end{equation}
for all real $x$ and $y$ and all $p\in[1,2]$. 
\big(As pointed out in \cite{clarkson}, inequality \eqref{eq:clarkson} obviously implies that $L^p$ is uniformly smooth, and in fact $p$-uniformly smooth.\big)
Actually, it is easy to see that Clarkson's inequality \eqref{eq:clarkson} is equivalent to the symmetric case of \eqref{eq:vBE}, with $K=1$. 

As mentioned in \cite{bahr65}, an inequality of the form \eqref{eq:vBE} is not of optimal order in $n$ for independent identically distributed real-valued zero-mean $X_i$'s 
and may be used together with a H\"older bound such as $\E|S_n|^p\le(\E S_n^2)^{p/2}$. Using similar considerations together with symmetrization and truncation, Manstavichyus \cite{manst82} obtained  bounds on $\E|S_n|^p$ from above and below, which differ from each other by an (unspecified) factor depending only on $p$. 
The proof of Theorem~\ref{th:} (and especially that of part (II) of Lemma~\ref{lem:x,Y}) shows that near-extremal r.v.'s $X_1,\dots,X_n$, for which the constant $C$ in \eqref{eq:} cannot be non-negligibly less than $C_f$, are as follows: $X_1$ and $X_2$ are independent, zero-mean, and both highly skewed in the same direction (both to the right or both to the left); $|X_2|$ is much smaller than $|X_1|$; and $X_3,\dots,X_n$ are zero or nearly so. This suggests that the inequality \eqref{eq:vBE} should be most useful for independent real-valued zero-mean $X_i$'s 
when the distributions of the $X_i$'s are quite different from one another and/or highly skewed and/or heavy-tailed. 

Again in the case when the differences $X_1,\dots,X_n$ are independent zero-mean r.v.'s, 
von Bahr and Esseen \cite{bahr65} made an effort to improve their constant $K=2-\frac1n$ in \eqref{eq:vBE}.  
For such $X_i$'s and the values of $p$ in a left neighborhood of $2$ such that  
$D(p):=\frac{13.52}{\pi(2.6)^p} \Ga(p) \sin\frac{\pi p}2
=\frac2\pi (\frac{13}5)^{2 - p} \Ga(p) \sin\frac{\pi p}2<1$, they showed that \eqref{eq:vBE} holds   
with 
$K=C_p^\vBE:=\frac1{\big(1-D(p)\big)_+}$, assuming the convention $\frac10:=\infty$; 
in fact, the constant factor $C_p^\vBE$ may improve on (i.e., may be less than) the factor $2-\frac1n$ only for  values of $p$ in a left neighborhood of $2$ such that $D(p)<\frac12$. It is stated (without proof) in \cite{bahr65} that $D(p)$ decreases in $p\in(1,2)$ and that the mentioned left neighborhood contains the interval $[1.6,2]$; cf. Figure~\ref{fig:tC_p}, where the von Bahr--Esseen constant factor $2\wedge C_p^\vBE$ is compared with the optimal (for \eqref{eq:}) constant factor $\tC_p$. 
\big(There are a couple of typos in \cite{bahr65}: in \cite[(11)]{bahr65}, one should have $r(2.6)^r$ instead of $(r2.6)^r$, and also the expression \cite[(12)]{bahr65} for $D(p)$ should have $\pi(2.6)^r$ instead of $(\pi 2.6)^r$.\big) 

\begin{wrapfigure}{l}{.45\textwidth}
\vspace*{-8pt}	
    \includegraphics[width=.45\textwidth]{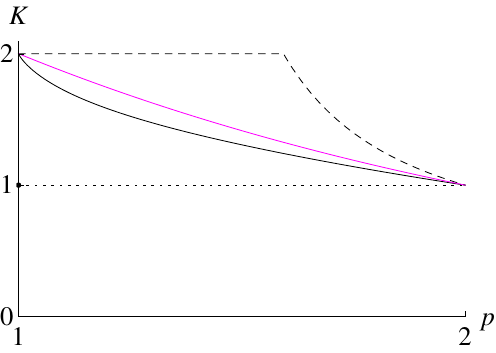}
  \caption{\ \ $\tC_p$, solid; \quad $W_p$, magenta;
	\qquad\ \protect\rule{0pt}{1pt}
	\ \qquad $2\wedge C_p^\vBE$, dashed;\quad $1$, dotted.}
\vspace*{-12pt}	
	\label{fig:tC_p}
\end{wrapfigure}

The method of \cite{bahr65} is based on a representation of the absolute moment $\E|X|^p$ of a r.v.\ $X$ as a certain integral transform of the Fourier transform of the distribution of $X$. More general representations, for the positive-part moments $\E X_+^p$, were obtained in \cite{brown70,positive}.  

Take now again any $p\in(1,2]$. 
Woyczy{\'n}ski \cite{woycz73} considered the class $\G_{p-1}$ of Banach spaces $\X$ defined by the following condition: there exist a map $G\colon\X\to\X^*$ and a constant $A=A_{p,\X}\in(0,\infty)$ such that for all $x$ and $y$ in $\X$ one has (i) $\|G(x)\|=\|x\|^{p-1}$, (ii) $G(x)x=\|x\|^p$, and (iii) $\|G(x)-G(y)\|\le A\|x-y\|^{p-1}$. The class $\G_1$ was introduced by Fortet and Mourier \cite{fort-mour}. 
Hoffmann-J{\o}rgensen \cite{hoff-jorg} proved that $\X\in\G_{p-1}$ iff $\X$ is $p$-uniformly smooth.

Woyczy{\'n}ski \cite{woycz73} showed that inequality \eqref{eq:vBE} holds for any independent zero-mean random vectors $X_1,\dots,X_n$ in any Banach space $\X\in\G_{p-1}$, 
with $|\cdot|$ and $K$ replaced by $\|\cdot\|$ and $A_{p,\X}$. 
As noted in \cite{woycz73}, the space $L^p$ is in $\G_{p-1}$, with the constant $A=2$; at that, one should take $G(x)=x^{[p-1]}:=|x|^{p-1}\sign x\in L^q=(L^p)^*$ for all $x\in L^p$. It is not hard to see that the best possible constant $A=A_{p,\X}$ for $\X=L^p$ is 
\begin{equation*}
	W_p:=\sup_{u\in(-1,1)}\frac{1-u^{[p-1]}}{(1-u)^{p-1}}=2^{2-p}, 
\end{equation*}
which is in agreement with the definition of $W_p$ in part (v) of Proposition~\ref{prop:C_powers}. 
Thus, one has \eqref{eq:vBE} with $K=W_p=2^{2-p}$ for independent zero-mean differences $X_1,\dots,X_n$, which may be either real-valued or, equivalently, with values in $L^p$ (in which case $|\cdot|$ is replaced by $\|\cdot\|_p$). 
The constant $K=W_p$ in \eqref{eq:vBE} is not the best possible one, even for independent zero-mean real-valued  $X_1,\dots,X_n$, even if $n$ is not fixed; indeed, by 
part (v) of Proposition~\ref{prop:C_powers}, $W_p>\tC_p$. 
On the other hand, the following proposition takes place.  

\begin{proposition}\label{prop:bw}\ 
One has $C_p^\vBE>W_p$ for all $p\in[1,2)$. 
\end{proposition} 

So, $C_p^\vBE>W_p>\tC_p$ for all $p\in(1,2)$. This comparison is illustrated in Figure~\ref{fig:tC_p}. 

\section{Proofs}\label{proofs} 
This section consists of four subsections. 
In Subsection~\ref{proofs:F12,C_f,C_psi}, we shall prove 5 propositions, of the 8 ones stated in Section~\ref{intro}; three of these 5 propositions will be used in the proof of Theorem~\ref{th:}, in Subsection~\ref{proof of th}.  
The proof of Proposition~\ref{prop:C_powers} (which is also used in the proof of Theorem~\ref{th:}) 
is more involved than those of the other propositions, 
and it will be presented separately, in Subsection~\ref{proof of C_powers}.  
Corollary~\ref{cor:concentr} and the related Propositions~\ref{lem:centring} and \ref{prop:minU} 
will be proved in Subsection~\ref{proof of {cor:concentr}}.

\subsection{Proofs of Propositions~\ref{prop:F12}, \ref{prop:osc}, \ref{prop:C_f}, \ref{prop:C_psi}, and \ref{prop:bw}}\label{proofs:F12,C_f,C_psi}

\begin{proof}[Proof of Proposition~\ref{prop:F12}]\ 
To begin, note that 
\begin{equation}\label{eq:psi'}
\psi'_t(x)=2(t\wedge x)	
\end{equation}
for all $x\in[0,\infty)$ and $t\in(0,\infty)$. 
Take any $f\in\F_{1,2}$. Then, by \eqref{eq:f''} and the right continuity of the monotonic right derivative $f''$ of $f'$, the relation \eqref{eq:ga} defines a nonnegative Borel measure $\ga=\ga_f$ on $(0,\infty]$ and, by Fubini's theorem,  
\begin{align}
	f'(x)&=\int_0^x f''(u)\dd u
	=2\int_0^x \dd u \int_{(u,\infty]}\ga(\dd t)
	=2\int_{(0,\infty]}\ga(\dd t) \int_0^{t\wedge x} \dd u \notag\\
	&=2\int_{(0,\infty]}(t\wedge x)\ga(\dd t) \label{eq:f'=}
\end{align}
for all $x\in[0,\infty)$. 
In particular, this proves part (III) of the proposition and
(taken with $x=1$) implies the condition $\int_{(0,\infty]}(t\wedge1)\ga(\dd t)<\infty$ in part (I) of the proposition. 
Further, for all $x\in[0,\infty)$ \eqref{eq:f'=} yields 
\begin{equation*}
		f(x)=\int_0^x f'(u)\dd u
			=2\int_0^x \dd u \int_{(0,\infty]}(t\wedge u)\ga(\dd t)
			=2\int_{(0,\infty]}\ga(\dd t) \int_0^x (t\wedge u) \dd u, 
\end{equation*}
which implies \eqref{eq:f}, since $\int_0^x (t\wedge u) \dd u=\frac12\,\psi_t(x)$ for all $x\in[0,\infty)$ and $t\in(0,\infty]$. 
This proves the ``only if'' implication in part (I) of the proposition, since the functions $f$ and $\psi_t$ are even. 

To prove the ``if'' implication, assume that \eqref{eq:f} holds for some nonnegative Borel measure $\ga$ on $(0,\infty]$ such that $\int_{(0,\infty]}(t\wedge1)\ga(\dd t)<\infty$ and for 
all $x\in\R$. In view of \eqref{eq:psi'}, the condition $\int_{(0,\infty]}(t\wedge1)\ga(\dd t)<\infty$ implies that the integral 
$\int_{(0,\infty]}\psi'_t(x)\ga(\dd t)$ converges uniformly over all $x$ in any given compact subset of the interval $(0,\infty)$. 
So, one finds that \eqref{eq:f} implies \eqref{eq:f'}, which in turn implies that $f'$ is nondecreasing and concave on $[0,\infty)$ \big(because the function $\psi_t$ is so, for each $t\in(0,\infty)$\big). 
It is also easy to see that $f\in C^1(\R)$, $f(0)=0$, and $f$ is even. 
Thus, it is checked that $f\in\F_{1,2}$, which completes the proof of the ``if'' implication in part (I) of the proposition. 

It remains to prove part (II). Take indeed any $f\in\F_{1,2}$. 
Take also any nonnegative Borel measure $\ga$ on $(0,\infty]$ such that $\int_{(0,\infty]}(t\wedge1)\ga(\dd t)<\infty$ and \eqref{eq:f} holds for all $x\in\R$. 
We have to show that \eqref{eq:ga} takes place for all $x\in(0,\infty)$. 
Take indeed any such $x$. 
Then, as has been shown, one has identities \eqref{eq:f'}. 
Therefore, for any $h\in(0,\infty)$  
\begin{equation}\label{eq:r_t}
	\frac12\,\frac{f'(x+h)-f'(x)}h
	=\int_{(0,\infty]}r_t(x,h)\ga(\dd t),    
\end{equation}
where
$r_t(x,h):=\tfrac1h\,\big[\big((x+h)\wedge t\big)-(x\wedge t)\big]$, which is bounded (between $0$ and $1$) and converges to $\ii{t>x}$ as $h\downarrow0$. So, \eqref{eq:ga} follows from \eqref{eq:r_t} by dominated convergence. 
This completes the proof of part (II) of the proposition as well. 
\end{proof}

\begin{proof}[Proof of Proposition~\ref{prop:osc}] 
Part (ii) of the proposition is obvious on recalling that $x_j=q^{2^{j-1}}-1$ for $j\in\intr1\infty$. 
Note also that $\rho\big((x_j+1)^{4/3}-1\big)=\frac43$ for $j\in\intr1\infty$. 
So, to prove then part (iii), it is enough to show that $\tp_\eff(r)$ decreases from $\frac53$ to $\frac32$ and then increases back to $\frac53$
as $r$ increases from $1$ to $\frac43$ and then to $2$, which follows because the expressions $2-\frac2{3r}$ and $1+\frac2{3r}$ are, respectively, increasing and decreasing in $r\in[1,2]$, and they are equal to each other at  $r=\frac43$. 

It remains to prove part (i) of the proposition, which is equivalent to 
\begin{equation}\label{eq:(i)}
f_\alt(x)=x^{\tp_\eff(r)+o(1)}
\end{equation}
as $x\to\infty$, where $r:=\rho(x)\in(1,2]$, so that $x=q^{r2^{j-1}}-1$. In other words, it suffices to prove that the convergence \eqref{eq:(i)} with $x=q^{r2^{j-1}}-1$ takes place uniformly in $r\in(1,2]$ 
as $j\to\infty$. 
Assume indeed that $j\to\infty$ and $x=q^{r2^{j-1}}-1$.  
Introduce $y_j:=x_j+1$, so that $y_j=q^{2^{j-1}}$ for $j=1,2,\dots$. 
Then $x=y_j^{r+o(1)}$, and uniformly over all $k\in\{0,\dots,j-1\}$ 
one has   
$x-\frac{1}{2}(x_k+x_{k+1})=x^{1+o(1)}$; moreover, if at that $k\to\infty$ then  $x_{k+1}-x_k=x_{k+1}^{1+o(1)}=y_k^{2+o(1)}$, which shows that 
the $k$th summand in the sum $\sum _{k=0}^{j-1}\dots$ in \eqref{eq:f_alt} is $(xy_k^{2-2/3})^{1+o(1)}=(y_j^r\,y_k^{4/3})^{1+o(1)}$ 
as $k\to\infty$. 
So, the sum $\sum _{k=0}^{j-1}\dots$ in \eqref{eq:f_alt} is 
$(y_j^r\,y_{j-1}^{4/3})^{1+o(1)}=(y_j^r\,y_j^{2/3})^{1+o(1)}=y_j^{r+\frac23+o(1)}$. 

To estimate the difference $x-x_j$, which appears on the right-hand side of \eqref{eq:f_alt}, we need to distinguish two possible cases: $r\in[1,\frac43)$ and $r\in[\frac43,2]$. 
Uniformly over all $r\in[\frac43,2]$ one has 
$x-x_j=x^{1+o(1)}=y_j^{r+o(1)}$, so that the term on the right-hand side of \eqref{eq:f_alt} before the sum $\sum _{k=0}^{j-1}\dots$ is $y_j^{2r-\frac23+o(1)}$, which yields  
$f_\alt(x)=y_j^{2r-\frac23+o(1)}+y_j^{r+\frac23+o(1)}
=y_j^{(2r-\frac23)\vee(r+\frac23)+o(1)}=y_j^{r\tp_\eff(r)+o(1)}=x^{\tp_\eff(r)+o(1)}$, as in \eqref{eq:(i)}. 

It remains to consider the values $r\in[1,\frac43)$. For such values of $r$, the relation $x-x_j=x^{1+o(1)}$ no longer holds; for instance, $x-x_j=0$ if $r=1$. However, in this case one can obviously write $0\le x-x_j\le x$ and also $\tp_\eff(r)=1+\frac2{3r}>2-\frac2{3r}$. 
So, the term on the right-hand side of \eqref{eq:f_alt} before the sum $\sum _{k=0}^{j-1}\dots$ is $\le y_j^{2r-\frac23+o(1)}\le y_j^{r+\frac23+o(1)}$, 
whereas still $\sum _{k=0}^{j-1}\dots=y_j^{r+\frac23+o(1)}$; 
so,   
$y_j^{r+\frac23+o(1)}\le f_\alt(x)\le y_j^{r+\frac23+o(1)}+y_j^{r+\frac23+o(1)}$, whence 
$f_\alt(x)=y_j^{r+\frac23+o(1)}=y_j^{r\tp_\eff(r)+o(1)}=x^{\tp_\eff(r)+o(1)}$, thus proving \eqref{eq:(i)} uniformly over all $r\in[1,\frac43)$ as well.   
\end{proof}

\begin{proof}[Proof of Proposition~\ref{prop:C_f}]\ 

\textbf{(i)}\quad 
Since the function $f$ is nonzero, the set $\supp\ga$ is a nonempty subset of $(0,\infty]$. So, $s_f=\inf\supp\ga\in[0,\infty]$. If $s_f=\infty$ then $\supp\ga=\{\infty\}$, which implies, in view of \eqref{eq:f}, that $f=\psi_\infty$, which contradicts the assumption on $f$ in Proposition~\ref{prop:C_f}.  
This proves part (i) of the proposition. 

\textbf{(ii)}\quad Take any $s\in(0,s_f]$ and $t\in\supp\ga$, so that 
$t\in[s_f,\infty]$. Then $s_f>0$ and 
it is straightforward to check that $L_{\psi_t;s}(x)=\psi_t(s)$ for any 
$x\in(0,s)$. Hence, by \eqref{eq:f} and \eqref{eq:f'}, 
$$L_{f;s}(x)=\int_{(0,\infty]}L_{\psi_t;s}(x)\ga(\dd t)
=\int_{(0,\infty]}\psi_t(s)\ga(\dd t)=f(s),$$
which proves part (ii) of Proposition~\ref{prop:C_f}.  

\textbf{(iii)}\quad 
Take any $s\in(s_f,\infty)$.  
Then $L'_{\psi_t;s}(0+)=2(s-t)_+$ for any $t\in(0,\infty]$. 
So, by \eqref{eq:f'} and \eqref{eq:ga}, 
$$L'_{f;s}(0+)=\int_{(0,\infty]}\!L'_{\psi_t;s}(0+)\ga(\dd t)
=2\int_{(0,\infty]}\!(s-t)_+\ga(\dd t)>0,$$
since for any $s\in(s_f,\infty)$ one has $\ga\big((0,s)\big)>0$. 
Similarly, 
$$L'_{f;s}(s-)=\int_{(0,\infty]}\!L'_{\psi_t;s}(s-)\ga(\dd t)
=-2\int_{(0,\infty]}\!t\ii{t<s}\ga(\dd t)<0.$$ 
This proves part (iii) of Proposition~\ref{prop:C_f}.

\textbf{(iv)}\quad 
In view of the rescaling identity $L_{f;s}(x)=L_{f_s;1}(\frac xs)$ with $f_s(u):=f(su)$, without loss of generality (w.l.o.g.) $s=1$. Then part (iv) of the proposition follows by parts (ii) and (iii) and the observation that $\ell_f(z):=L_{f;1}(1-\sqrt z)$ is concave in $z\in(0,1)$. In view of \eqref{eq:f}, it is enough to prove this observation for $f=\psi_t$ with $t\in(0,\infty]$; at that, by part (ii) of Proposition~\ref{prop:C_f} and because $s_{\psi_t}=t$, w.l.o.g.\ let us assume that $0<t<s=1$. Observe that the second derivative $\ell''_{\psi_t}(z)$ in $z$ admits of a piecewise-algebraic expression, which may be quickly obtained by using the Mathematica command \verb9PiecewiseExpand9. 
Applying then a \verb9Reduce9 command, one finds that $\ell''_{\psi_t}(z)\le0$ for all $t\in(0,1)$ and  $z\in(0,1)$. 
Now part (iv) of Proposition~\ref{prop:C_f} follows. 

\textbf{(v)}\quad 
Part (v) of the proposition follows by parts (i)--(iv), on recalling \eqref{eq:C_f} 
and taking into account that $L_{f;s}(0+)=f(s)$, for all $s\in(0,\infty)$. 

Proposition~\ref{prop:C_f} is now completely proved. 
\end{proof}

\begin{proof}[Proof of Proposition~\ref{prop:C_psi}]\ 
Take any $t\in(0,\infty]$. 
That $C_{\psi_\infty}=1$ follows immediately by \eqref{eq:C_f}. 
So, w.l.o.g.\ $t\in(0,\infty)$, and then, by \eqref{eq:C_f} and homogeneity, w.l.o.g.\ $t=1$. 
Thus, it remains to show that $C_{\psi_1}=2$. 
Take any $s\in(s_{\psi_1},\infty)=(1,\infty)$ and observe that $L'_{\psi_1;s}(1)=-2(s\wedge2)<0$, whereas $L'_{\psi_1;s}(1-)=-2(s\wedge2)+2s\ge0$. Therefore, by part (iv) of Proposition~\ref{prop:C_f}, $\max\nolimits_{x\in(0,s)}L_{\psi_1;s}(x)=L_{\psi_1;s}(1)=s^2 - (s - 2)_+^2$. 
Now, using part (v) of Proposition~\ref{prop:C_f}, it is easy to see 
that 
$C_{\psi_1}=\sup\nolimits_{s\in(1,\infty)}\tfrac{s^2 - (s - 2)_+^2}{s^2 - (s -1)_+^2} 	
	=\lim\nolimits_{s\to\infty}\tfrac{s^2 - (s - 2)_+^2}{s^2 - (s -1)_+^2}=2$. 
\end{proof}

\begin{proof}[Proof of Proposition~\ref{prop:bw}] 
Take any $p\in[1,2)$. 
It suffices to show that 
\begin{equation}\label{eq:be<1}
	\be(p):=\big(1-D(p)\big)2^{2-p}\overset{\text{(?)}}<1. 
\end{equation}
Observe that 
\begin{align*}
\be'(p)&=-2^{2-p} \ln2+
(\tfrac{26}5)^{2-p} \,\tfrac{\Gamma (p)}\pi 
\big[2 (\sin\tfrac{\pi  p}2) \big(\ln\tfrac{26}5-(\ln\Ga)'(p)\big)-\pi  \cos\tfrac{\pi  p}2\big]  \\
&>-2^{2-p} \ln2>-2 \ln2>-1.4;
\end{align*}
the first inequality here follows because $\cos\tfrac{\pi  p}2\le0$, $\sin\tfrac{\pi  p}2>0$, and $\ln\tfrac{26}5-(\ln\Ga)'(p)\ge\ln\tfrac{26}5-(\ln\Ga)'(2)>0$, taking into account that $\ln\Ga$ is convex and hence $(\ln\Ga)'$ is increasing. 
It is easy to see that $\max\{\be(1+\frac i4)\colon i\in\intr12\}<1-0.49$. 
So, $\be(p) < \be(1+\frac i4) + (1.4)\frac14<1-0.49 + (1.4)\frac14<1$ for $p\in[1 + \frac{i - 1}4, 1 + \frac i4]$ and $i\in\intr12$; thus, \eqref{eq:be<1} holds for all $p\in[1,\frac32]$. 

Next,
\begin{equation*}
	\be_2(p):=25 \pi \,\be''(p)\,2^{p - 1} 
	=A+B(E_1+E_2+E_3+E_4), 
\end{equation*}
where
\begin{gather*}
	A := 50\pi \, \ln^2 2, \quad B := 169\,\Ga(p)\, (\tfrac5{13})^p , \\
	E_1:= 4 \pi\,(\cos\tfrac{\pi  p}2)\, \ln\tfrac{26}5, \quad 
	E_2:= \ka\,\sin\tfrac{\pi  p}2, \\ 
	E_3:= -4 \big((\ln\Ga)'(p)^2 + (\ln\Ga)''(p)\big) \sin\tfrac{\pi  p}2, \\ 
	E_4:= (\ln\Ga)'(p) \big(-4 \pi \cos\tfrac{\pi  p}2 + 8 \ln\tfrac{26}5\, \sin\tfrac{\pi  p}2\big), 
\end{gather*}
and $\ka:=\pi ^2-4 \ln^2 2-4 \ln^2\frac{13}5-8 \ln2\, \ln\frac{13}5<0$, whence $E_2<0$.  
Also, $E_3<0$, because $(\ln\Ga)''>0$. 
Let us next bound $E_1$ and $E_4$ from above, assuming that $p\in[\frac32,2]$. 
Then $E_1\le4 \pi\,(\cos(\pi\frac34)\, \ln\tfrac{26}5<-14.6$;  
also, $(\ln\Ga)'(p)\ge(\ln\Ga)'(\frac32)>0$ and $(\ln\Ga)'(p)\le(\ln\Ga)'(2)$, so that  
$E_4\le(\ln\Ga)'(2) \big(4 \pi + 8 \ln\tfrac{26}5\big)<10.9$.  
Thus, for all $p\in[\frac32,2]$ 
\begin{equation*}
	\be_2(p)\le50\pi\, \ln^2 2+169\,\Ga(\tfrac32)\, (\tfrac5{13})^2(-14.6+10.9)<-6<0
\end{equation*}
and hence $\be''(p)<0$, so that $\be$ is strictly concave on $[\frac32,2]$. 
At that, $\be(2)=1$ and $\be'(2)=1-\ln2>0$; so, \eqref{eq:be<1} holds for all $p\in[\frac32,2)$ as well. 
\end{proof}

\subsection{Proof of Proposition~\ref{prop:C_powers}}\label{proof of C_powers} 
Of the 5 parts of the proposition, the most difficult to prove are parts~(iii) and (v), which are based to a certain extent on several lemmas. To state these lemmas, we need more notation. Recall the definition of $\ell(p,x)$ in \eqref{eq:ell} and introduce 
\begin{gather*}
	\ell_p(p,x):=\pd{}p\ell(p,x), \quad
	\ell_x(p,x):=\pd{}x\ell(p,x), \\
	\ell_{x,x}(p,x):=\pd{}x\ell_x(p,x)
	=\pd{{}^2}{x^2}\ell(p,x) 
\end{gather*}
and also 
\begin{equation*}
	p^*_x:=\tfrac14(25 x + 2)\quad\text{and}\quad 
	x^*_p:=\tfrac2{25}\,(2 p-1),  
\end{equation*}
so that $x=x^*_p\iff p=p^*_x$. 
Now we are ready to state the lemmas: 

\begin{lemma}\label{lem:lxx}
For all $p\in(1,2)$ and $x\in(0,\frac12)$, one has $\ell_{x,x}(p,x)<0$ and hence $\ell_{x,x}(p,x)\ne0$. 
\end{lemma}

\begin{lemma}\label{lem:B>0}
For all $p\in(1,2)$, 
\begin{equation}\label{eq:B>0}
		  B(p):=4 (p-1)^{p-1}-(6-p)^{p-1}>0. 
\end{equation}
\end{lemma}

\begin{lemma}\label{lem:DLx}
For all $p\in(1,2)$ and $x\in(0,\frac12)$ such that $x\ge x^*_p$, one has $\ell_x(p,x)<0$. 
\end{lemma}

\begin{lemma}\label{lem:DLp}
For all $p\in(1,2)$ and $x\in(0,\frac12)$ such that $x<x^*_p$, one has $\ell_p(p,x)<0$. 
\end{lemma} 

The proofs of these lemmas are deferred to the end of this subsection. 
Let us now consider the four parts of Proposition~\ref{prop:C_powers}. 



\textbf{(i,ii)}\quad Take any $p\in(1,2)$. 
Observe that $\ell_x(p,\frac{p-1}2)=2^{1-p} \big((p-1)^{p-1}-\break
(3-p)^{p-1}\big) p<0$, 
since $p-1<3-p$. 
On the other hand, $\ell_x(p,\frac{p-1}5)=5^{1 - p} p B(p)>0$, 
by Lemma~\ref{lem:B>0}. 
So, any value of $x_{f;s}$ as in part (iv) of Proposition~\ref{prop:C_f} (for $f=|\cdot|^p$) must be in the interval $(\frac{p-1}5,\frac{p-1}2)\subset(0,\frac12)$. By Lemma~\ref{lem:lxx} and part (iii) of Proposition~\ref{prop:C_f} (with $s_f=0$), $\ell_x(p,x)$ is strictly decreasing in $x\in(0,\frac12)$ from a positive value to a negative one. 
Now, in view of part (v) of Proposition~\ref{prop:C_f}, parts (i) and (ii) of Proposition~\ref{prop:C_powers} 
follow, taking also into account that the equation \eqref{eq:x_p eq} is equivalent to $\ell_x(p,x)=0$. 

\textbf{(iii)}\quad 
By part (i) of Proposition~\ref{prop:C_powers}, $x_p$ is the only root $x\in(0,\frac12)$ of the equation $\ell_x(p,x)=0$, for each $p\in(1,2)$. So, 
by Lemma~\ref{lem:lxx} and the implicit function theorem, $\tC_p$ is differentiable, and even real-analytic, and hence continuous in $p\in(1,2)$.  

Next, by Lemma~\ref{lem:DLx}, for any $p\in(1,2)$ and $x\in(0,\frac12)$ the equality $\ell_x(p,x)=0$ implies $x<x^*_p$, which in turn implies $\ell_p(p,x)<0$, by Lemma~\ref{lem:DLp}. 
So, for any $p\in(1,2)$ one has $\ell_p(p,x_p)<0$, whence 
$\fd{}p\tC_p=\fd{}p\ell(p,x_p)=\ell_p(p,x_p)+\ell_x(p,x_p)\pd{}p x_p=\ell_p(p,x_p)<0$, which verifies that $\tC_p$ is decreasing in $p\in(1,2)$.  

Thus, to complete the proof of part (iii) of the proposition, it remains to show that $\tC_{1+}=2$ and $\tC_{2-}=1$ (recall that $\tC_2=1$, by \eqref{eq:tC_2}). 
Here, consider first the case $p\downarrow1$. Observe that then $\ell(p-1,p)=(2-p)^p-(p-1)^p+p(p-1)^{p-1}\to2$; on the other hand, by \eqref{eq:1<C_f<2}, $\tC_p\le2$ for all $p\in(1,2]$. It indeed follows that $\tC_{1+}=2$. 
Next, for all $x\in(0,1)$ and $p\in(\frac32,2)$, one has 
$\ell(2,x)=1$ and 
$|x^p \ln x|<|x^{p-1} \ln x|<|x^{1/2} \ln x|<\frac2e<1$, whence 
$|\ell_p(p,x)|=|x^{p-1}+p x^{p-1} \ln x-x^p \ln x+(1-x)^p \ln (1-x)|
\le|x^{p-1}|+|p x^{p-1} \ln x|+|x^p \ln x|+|(1-x)^p \ln (1-x)|
\le1+2+1+1=5$; so, letting $p\uparrow2$, one has $\ell(p,x)=\ell(2,x)-\int_p^2\ell_p(r,x)\dd r\le1+5(2-p)\to1$, whence $\limsup_{p\uparrow2}\tC_p=\limsup_{p\uparrow2}\ell(p,x_p)\le1$. 
It remains to refer, again, to \eqref{eq:1<C_f<2}. 

\textbf{(iv)}\quad 
The proof of part (iv) of the proposition is straightforward.  



\textbf{(v)}\quad The equalities $\tC_{1+}=W_{1+}$ and $\tC_2=\tC_{2-}=W_{2-}=W_2$, and the similar equalities for the upper and lower bounds $\tC_p^{-,1}$, $\tC_p^{-,2}$, $\tC_p^{+,1}$, and $\tC_p^{+,2}$ on $\tC_p$ follow immediately by part (iii) of the proposition. 
Take now any $p\in(1,2)$. 
Consider $\tl(p,z):=\ell(p,1-\sqrt z\,)$, where $z\in(0,1)$. 
By parts (i) and (ii) of Proposition~\ref{prop:C_powers}, 
\begin{equation*}
	\tC_p=\max_{z\in(0,1)}\tl(p,z)=\max_{z\in(z_1,z_2)}\tl(p,z),
\end{equation*}
where $z_1:=z_1(p):=(\frac{3-p}2)^2$ and $z_2:=z_2(p):=(\frac{6-p}5)^2$ \big(since the values $\frac{p-1}2$ and $\frac{p-1}5$ of $x$ correspond, respectively, to the values $z_1$ and $z_2$ of $z$ under the correspondence given by the formula $x=1-\sqrt z$.\big) 
Hence, $\tC_p>\tl(p,z_1)\vee\tl(p,z_2)=\tC_p^{-,1}\vee\tC_p^{-,2}$, which proves the first inequality in \eqref{eq:<tC_p<}. 
It follows from the proof of part (iv) of Proposition~\ref{prop:C_f} that $\tl(p,z)$ is concave in $z\in(0,1)$. 
Also, in the proof of parts (i) and (ii) of the proposition it was observed that 
$\ell_x(p,\frac{p-1}5)>0>\ell_x(p,\frac{p-1}2)$, which is equivalent to $\tl_z(p,z_2)<0<\tl_z(p,z_1)$, where $\tl_z:=\pd\tl z$.  
Therefore, $\tl(p,z)\le\tl(p,z_1)+\tl_z(p,z_1)(z-z_1)<\tl(p,z_1)+\tl_z(p,z_1)(z_2-z_1)=\tC_p^{+,1}$ and $\tl(p,z)\le\tl(p,z_2)+\tl_z(p,z_2)(z-z_2)<\tl(p,z_2)+\tl_z(p,z_2)(z_1-z_2)=\tC_p^{+,2}$ for all $z\in(z_1,z_2)$, which yields the second inequality in \eqref{eq:<tC_p<}. 
The third inequality in \eqref{eq:<tC_p<} is trivial.  

So, it remains to prove the last inequality in \eqref{eq:<tC_p<}. 
It is enough to show that $\rho(p)<0$, where 
\begin{align*}
	\rho(p)&:=
	2\times5^p\big(\tC_p^{+,2}-2^{2-p}\big) \\
	&=A(p)+\tfrac34\,\tfrac{27-7p}{6-p}\,p(p-1)B(p), \\
	A(p)&:=10 p (p-1)^{p-1}-2 (p-1)^p-2^{3-p} 5^p+2 (6-p)^p, 
\end{align*}
and $B(p)$ is as in \eqref{eq:B>0}. 
Observe next that $27-7p\le\frac{49}{60}(6-p)^2$. Hence and in view of Lemma~\ref{lem:B>0}, 
\begin{equation*}
4\rho(p)\le\trho(p):=4A(p)+\tfrac{49}{20}(6-p)p(p-1)B(p); 	
\end{equation*}
thus, it suffices to show that $\trho(p)<0$, which can be rewritten as $\hat\rho(r)<0$ for $r\in(0,\frac25)$, where 
\begin{equation*}
\hat\rho(r):=16(\tfrac25)^{1+\frac52\,r}\trho(1+\tfrac52\,r). 
\end{equation*}
One has 
\begin{equation*}
	\rho_1(s):= 
	\hat\rho'(r) 
	\frac{(1 + s)^3}{r^{5 r/2}}
	=A_1(s)+4 B_1(s) s^{\frac{5}{s+1}}, 
\end{equation*}
where 
\begin{align*}
	A_1(s) &:= 16 (-62 + 2202 s + 1160 s^2 + 121 s^3) + 
  80 (40 + 382 s + 105 s^2 + 8 s^3) \ln\tfrac2{1 + s}, \\
	B_1(s) &:= 1572 - 367 s - 795 s^2 - 
  81 s^3 + (-1310 s + 75 s^2 + 160 s^3) \ln\tfrac{2s}{1 + s}, 
\end{align*}
and $s:=\frac2r-1$, so that $r=\frac2{1+s}$, and $r\in(0,\frac25)$ iff $s>4$.  
Using a \verb9Reduce9 command, one finds that $B_1(s)$ switches in sign from $-$ to $+$ as $s$ increases from $4$ to $\infty$, and the switch occurs at a certain point $s_*=31.4\dots$.  
With 
\begin{equation*}
	\trho_1(s):=\frac{\rho_1(s)}{s^{5/(1 + s)} B_1(s)}
	=\frac{A_1(s)}{s^{5/(1 + s)} B_1(s)}+4, 
\end{equation*}
another \verb9Reduce9 command shows (in about 12 sec) that 
\begin{equation*}
	\rho_2(s):=\trho'_1(s) B_1(s)^2 s^{(6 + s)/(1 + s)} \tfrac{(1 + s)^2}{80} 
\end{equation*}
switches in sign from $+$ to $-$ to $+$ to $-$ as $s$ increases from $4$ to $\infty$, and the switches occur at certain points $s_1=5.2\dots$, $s_2=21.5\dots$, and $s_3=42.7\dots$.  
So, $\trho_1(s)$ switches from increase to decrease to increase as $s$ increases from $4$ to $s_1=5.2\dots$ to $s_2=21.5\dots$ to $s_*=31.4\dots$, and then $\trho_1(s)$ switches from increase to decrease as $s$ increases from $s_*=31.4\dots$ to $s_3=42.7\dots$ to $\infty$. 
Next, $\trho_1(s)<0$ for $s\in\{4,s_1,s_2,s_3\}$; also, $\rho_1(s_*)<0$, whence $\trho_1(s_*-)=\infty>0$ and $\trho_1(s_*+)=-\infty<0$ (on recalling the definitions of $\trho_1(s)$ and $s_*$). 
It follows that $\trho_1(s)$ switches in sign from $-$ to $+$ as $s$ increases from $4$ to $s_*$, and $\trho_1<0$ on $(s_*,\infty)$. 
Therefore, $\rho_1(s)$ switches in sign from $+$ to $-$ as $s$ increases from $4$ to $\infty$. 
Equivalently, $\hat\rho'(r)$ switches in sign from $-$ to $+$ as $r$ increases from $0$ to $\frac25$. 
This implies that $\hat\rho(r)$ switches from decrease to increase as $r$ increases from $0$ to $\frac25$. 
Equivalently, $(\tfrac25)^p\trho(p)$ switches from decrease to increase as $p$ increases from $1$ to $2$. 
Note also that $\trho(1+)=\trho(2-)=\trho(2)=0$. 
So, indeed $\trho(p)<0$, for all $p\in(1,2)$. 
This proves part (v) and thus the entire proposition, modulo Lemmas~\ref{lem:lxx}--\ref{lem:DLp}. 


\begin{proof}[Proof of Lemma~\ref{lem:lxx}]\ 
Introduce the new variable $y:=\frac{1 - x}x$, so that $y>1$ for $x\in(0,\frac12)$. Then, for any $p\in(1,2)$ and $x\in(0,\frac12)$, 
\begin{align*}
	\ell_{x,x}(p,x)\,\frac{(1 - x)^{2 - p}}{p(p-1)}&=1-(2-p) y^{3-p}-(3-p) y^{2-p} \\
	&<1-(2-p) -(3-p) =2(p-2)<0, 
\end{align*}
which proves the lemma. 
\end{proof}

\begin{proof}[Proof of Lemma~\ref{lem:B>0}]\ 
Take indeed any $p\in(1,2)$. Note that \eqref{eq:B>0} is equivalent to 
$
	\tilde B(p):=\ln \left(4 (p-1)^{p-1}\right)-\ln \left((6-p)^{p-1}\right)>0. 
$ 
Next, $\tilde B'(p)=1+r+\ln r$, where $r:=\frac{p-1}{6-p}$, so that $\tilde B'(p)$ is increasing in $p$, and $\tilde B'(2)<0$, which implies that $\tilde B'(p)<0$ and hence $\tilde B(p)$ is decreasing in $p$, with $\tilde B(2)=0$. Thus, indeed $\tilde B(p)>0$.  
\end{proof}

\begin{proof}[Proof of Lemma~\ref{lem:DLx}]\ 
Throughout the proof, it is assumed that indeed $p\in(1,2)$ and $x\in(0,\frac12)$. 
Let 
\begin{equation*}
	(D_x\ell)(p,x):=\frac{\ell_x(p, x)}{p (1 - x)^{p-1}}, 
\end{equation*}
so that $D_x\ell$ equals $\ell_x$ in sign. 
Then $\pd{}x (D_x\ell)(p,x)=(p-2) (p-1) (1-x)^{-p} x^{p-3}<0$, so that $(D_x\ell)(p,x)$ decreases in $x$. 
Consider now 
\begin{equation*}
	H(p):=(D_x\ell)(p,x^*_p)=(27-4 p)^{1-p} (4 p-2)^{p-2} (21 p-23)-1. 
\end{equation*}
Obviously, $H(p)<0$ for $p\le\frac{23}{21}$. Let us show that $H(p)<0$ for $p\in(\frac{23}{21},2)$ as well. 
Observe that 
\begin{multline*}
	H'(p)\frac{4 (27-4 p)^{p-1} (2 p-1)^2 (4 p-2)^{-p}}{21 p-23} \\
	=H_1(p):=\frac{25 \left(42 p^2-92 p+73\right)}{(27-4 p) (2 p-1) (21 p-23)}+\ln
   \frac{4 p-2}{27-4 p}.  
\end{multline*}
Using the Mathematica command \verb9Minimize9, one finds that $H_1(p)>0$ and hence $H'(p)>0$ for $p\in(\frac{23}{21},2]$. Since $H(2)=0$, it indeed follows that $H(p)<0$ for $p\in(\frac{23}{21},2)$ and thus for all $p\in(1,2)$. 
So, one has $(D_x\ell)(p,x^*_p)<0$. Recalling that $(D_x\ell)(p,x)$ decreases in $x$, one has $(D_x\ell)(p,x)<0$ or, equivalently, $\ell_x(p,x)<0$ --- provided that $x\ge x^*_p$. 
\end{proof}

\begin{proof}[Proof of Lemma~\ref{lem:DLp}]\ 
Throughout the proof, it is assumed that indeed $p\in(1,2)$ and $x\in(0,\frac12)$. 
Let 
\begin{align*}
	(D_p\ell)(p,x)&:=\frac{\ell_p(p, x)}{-(1 - x)^p \ln(1 - x)}
	=\frac{x^{p-1} (1+(p-x) \ln x)}{-(1-x)^p\,\ln (1-x)}-1, \\
	(D_pD_p\ell)(p,x)&:=\pd{(D_p\ell)(p,x)}p\,\frac{(1-x)^p}{x^{p-1}}\,\frac{\ln (1-x)}{\ln x}, 
\end{align*}
so that $D_p\ell$ and $D_pD_p\ell$ equal $\ell_p$ and $\pd{(D_p\ell)}p$ in sign, respectively. 
Then \break
$\pd{}p (D_pD_p\ell)(p,x)=\ln(1-x)-\ln x>0$ \big(since $x\in(0,\frac12)$\big), so that $(D_pD_p\ell)(p,x)$ increases in $p$. 
Consider now 
\begin{equation*}
	(D_pD_p\ell)(p^*_x,x)
	=\frac{[4+(21 x+2) \ln x]\ln (1-x) - [8+(21 x+2) \ln x]\ln x}{4 \ln x}. 
\end{equation*}
Observe that $1<p^*_x<2 \iff \frac2{25} < x < \frac6{25}$, and then use the Mathematica command \verb9Reduce9 to find that $(D_pD_p\ell)(p^*_x,x)>0$ provided that $\frac2{25} < x < \frac6{25}$. 
Similarly, $(D_pD_p\ell)(1,x)>0$ provided that $0<x\le\frac2{25}$. 
Thus, $(D_pD_p\ell)(1\vee p^*_x,x)>0$ for all $x\in(0,\frac6{25})$. 
Recalling that $(D_pD_p\ell)(p,x)$ increases in $p$, one has $(D_pD_p\ell)(p,x)>0$ for all $p\in[1\vee p^*_x,2)$. 
It follows that $(D_p\ell)(p,x)$ increases in $p\in[1\vee p^*_x,2)$. 
Now use \verb9Reduce9 to check that $(D_p\ell)(2,x)<0$, which yields $(D_p\ell)(p,x)<0$ or, equivalently, 
$\ell_p(p,x)<0$ for $p\in[1\vee p^*_x,2)$ or, equivalently, for $x\le x^*_p$. 
\end{proof}

\subsection{Proofs of Corollary~\ref{cor:concentr} and Propositions~\ref{lem:centring} and \ref{prop:minU}}\label{proof of {cor:concentr}} 

First in this subsection we shall prove Proposition~\ref{prop:minU}, then Proposition~\ref{lem:centring}, and finally Corollary~\ref{cor:concentr}. 

\begin{proof}[Proof of Proposition~\ref{prop:minU}]
The convexity of $U_f(c,s,a)$ in $a\in\R$ follows immediately from that of $f$. 
Since $f'$ is strictly positive and nondecreasing on $(0,\infty)$, it follows that $f(\infty-)=\infty$; similarly (or because $f$ is even), $f(-\infty+)=\infty$. So, $U_f(c,s,a)\to\infty$ as $|a|\to\infty$. 
Therefore and by continuity, there is a minimizer of $U_f(c,s,a)$ in $a\in\R$. 
Take any such minimizer, say $a_*$. 
Since $f\in\C^1(\R)$, the partial derivative of $U_f(c,s,a)$ in $a$ at $a=a_*$ is $0$; that is, $cf'(s-c+a_*)+(s-c)f'(a_*-c)=0$, which can be rewritten as 
\begin{equation}\label{eq:cf'=..}
	cf'(s-c+a_*)=(s-c)f'(c-a_*),
\end{equation}
since $f$ is even and hence $f'$ is odd. Recall also that $f'$ is strictly positive and hence nowhere zero on $(0,\infty)$. It follows that the arguments $s-c+a_*$ and $c-a_*$ of $f'$ in \eqref{eq:cf'=..} must be of the same sign; noting that the sum of these arguments is $s>0$, one concludes that they must be both positive; equivalently, $a_*\in(c-s,c)$.  
Moreover, $f'$ is positive and nondecreasing on $(0,\infty)$ and 
$0<c<s-c$, so that \eqref{eq:cf'=..} yields $f'(s-c+a_*)>f'(c-a_*)$ and hence 
\begin{equation}\label{eq:s-c+a>}
	s-c+a_*>c-a_*.
\end{equation}

If a minimizer of $U_f(c,s,a)$ in $a$ is not unique, then the first two partial derivatives of $U_f(c,s,a)$ in $a$ are identically zero for all $a$ in some nonempty open interval $(a_1,a_2)\subset(c-s,c)$. That is, 
$cf'(s-c+a)=(s-c)f'(c-a)$ and $cf''(s-c+a)+(s-c)f''(a-c)=0$ for all $a\in(a_1,a_2)$. Since $f''$ is nonnegative and even, it follows that $f''(c-a)=f''(a-c)=0$ for all $a\in(a_1,a_2)$, so that $f''=0$ on the interval $(c-a_2,c-a_1)$. Because $a_2\le c$ and $f''$ is nonnegative and nonincreasing on $(0,\infty)$, one has $f''=0$ on the interval $(c-a_2,\infty)$, so that $f'$ is constant on the same interval. 
On recalling \eqref{eq:s-c+a>}, 
one has 
$s-c+a>c-a>c-a_2$ for any $a\in(a_1,a_2)$, which shows that $f'(s-c+a)=f'(c-a)$; however, this contradicts the previously obtained inequality $f'(s-c+a_*)>f'(c-a_*)$ for any minimizer $a_*$. 

Next, the formula \eqref{eq:a_t,c,s} for the unique minimizer of $U_{\psi_t}(c,s,a)$ in $a$ 
is easy to verify by noting that the partial derivative of $U_{\psi_t}(c,s,a)$ in $a$ at $a=\tfrac c{s-c}\,(s-c-t)_+$ is $0$. 
Moreover, for any real $c$ an $t$ such that $c>t>0$ 
one has 
$\dfrac{U_{\psi_1}(c,s,0)}{U_{\psi_1}(c,s,a_{\psi_1;c,s})}\underset{s\to\infty}\longrightarrow2 -\frac t{2 c}$, and then $2 -\frac t{2 c}\underset{c\to\infty}\longrightarrow2$, which shows that $\ka_{\psi_t}=2$. 

It remains to prove that the unique minimizer $a=a_{f;c,s}$ is nonnegative. Equivalently, it remains to show that the partial derivative of $U_f(c,s,a)$ in $a$ is no greater than $0$ at $a=0$, that is,  
\begin{equation}\label{eq:der at 0}
	cf'(s-c)\ge(s-c)f'(c). 
\end{equation}
By the linearity relation \eqref{eq:f'} and homogeneity, w.l.o.g.\ $f=\psi_t$ for some $t\in(0,\infty)$, in which case \eqref{eq:der at 0} is equivalent to $a_{\psi_t;c,s}\ge0$, and that is obvious from \eqref{eq:a_t,c,s}. 
\end{proof} 

\begin{proof}[Proof of Proposition~\ref{lem:centring}]
Take indeed any $f\in\F_{1,2}\setminus\{0\}$. 
By e.g.\ \cite[Proposition~3.18]{disintegr}, any zero-mean probability distribution on $\R\setminus\{0\}$ is a mixture of zero-mean probability distributions on 2-point sets. 
Therefore, w.l.o.g.\ the zero-mean r.v.\ $X$ takes on only two values, so that $X=X_{c,d}$, where $c$ and $d$ are positive real numbers, and $X_{c,d}$ is a r.v.\ such that $\P(X_{c,d}=-c)=\frac d{c+d}$ and $\P(X_{c,d}=d)=\frac c{c+d}$. 
Take now any $c$ and $s$ such that 
$0<c<s<\infty$, and introduce 
\begin{equation}\label{eq:R_f}
R_f(c,s,a):=\frac{U_f(c,s,0)}{U_f(c,s,a)}=\frac{\E f(X_{c,s-c})}{\E f(X_{c,s-c}+a)}.  
\end{equation}
So, the best constant $\ka$ in \eqref{eq:centring} is given by a formula similar to \eqref{eq:ka_f}, but with the restrictions $c\in(0,s)$ and $a\in\R$ instead of $c\in(0,\tfrac s2)$ and $a\in(0,c)$. 
That $c\in(0,s)$ can be reduced to $c\in(0,\tfrac s2)$ follows by the symmetry relation $R_f(c,s,a)\equiv R_f(s-c,s,-a)$ and the continuity of $R_f(c,s,a)$ in $c$. 
Finally, the condition $a\in\R$ can be reduced to $a\in(0,c)$ by Proposition~\ref{prop:minU} and the continuity of $R_f(c,s,a)$ in $a$.  
\end{proof} 

\begin{proof}[Proof of Corollary~\ref{cor:concentr}]\

\text{(I)}\quad 
Take indeed any $f\in\F_{1,2}\setminus\{0\}$.  
Consider the martingale expansion 
$$ Y = \E Y+\xi_1+\dots+\xi_n$$
of $Y$
with the martingale-differences 
\begin{equation}\label{eq:xi_i}
\xi_i:= \E_i Y - \E_{i-1} Y
\end{equation}
for $i\in\intr1n$, 
where $\E_i$ stands for the conditional expectation given the $\si$-algebra generated by $(X_1,\dots,X_i)$, with $\E_0:=\E$.
For each $i\in\intr1n$ introduce the r.v.\ 
$\eta_i :=\E_i(Y - \tilde Y_i)$, where $\tilde Y_i := g(X_1,\dots,X_{i-1},x_i,X_{i+1},\dots,X_n)$;  
then, in view of \eqref{eq:Lip} or \eqref{eq:LipE}, $|\eta_i|\le\rho_i(X_i,x_i)$;  
because $f(u)$ is increasing in $|u|$, it follows that $f(\eta_i)\le f\big(\rho_i(X_i,x_i)\big)$ and hence $\E f(\eta_i)\le\E f\big(\rho_i(X_i,x_i)\big)$; also, $\xi_i=\eta_i-\E_{i-1}\eta_i$, since the r.v.'s $X_1,\dots,X_n$ are independent.  
Now \eqref{eq:concentr} follows from Theorem~\ref{th:} and Proposition~\ref{lem:centring}, which latter yields $\E_{i-1}f(\xi_i)\le\ka_f\E_{i-1}f(\eta_i)$ and hence $\E f(\xi_i)\le\ka_f\E f(\eta_i)$.   

To check the inclusion $\ka_f\in[1,2]$ in \eqref{eq:ka_f}, note first that the inequality $\ka_f\ge1$ follows by the continuity of $U_f(c,s,a)$ in $a$, at $a=0$. As for the inequality $\ka_f\le2$, it can be rewritten as 
\begin{equation}\label{eq:U<2U}
U_f(c,s,0)\le2U_f(c,s,a)	
\end{equation}
for all $s\in(0,\infty)$, $c\in(0,\tfrac s2)$, and $a\in(0,c)$, where w.l.o.g.\ $f=\psi_t$ \big(for some $t\in(0,\infty)$, by \eqref{eq:U_f} and \eqref{eq:f}\big) and $s=1$ (by homogeneity). 
Take then indeed any $c\in(0,\tfrac12)$ and $a\in(0,c)$. 
By Proposition~\ref{prop:minU}, w.l.o.g.\ $a=a_{\psi_t;c,1}$. 
Using a \verb9Simplify9 Mathematica command for $U_{\psi_t}(c,1,a_{\psi_t;c,1})$ and then following with a \verb9Reduce9, 
one quickly verifies that \eqref{eq:U<2U} indeed holds for $f=\psi_t$.  
This completes the proof of part (I) of Corollary~\ref{cor:concentr}. 

\text{(II)}\quad To obtain the expression in \eqref{eq:tka_p} for $\tka_p=\ka_{|\cdot|^p}$, note first that, by homogeneity of the power function $f=|\cdot|^p$, w.l.o.g.\ $s=1$. Then
solve the equation \eqref{eq:cf'=..} to find the unique minimizer 
\begin{equation}
	a_*=\ta_{p;c}:=
	a_{|\cdot|^p;c}=c-\frac{c^{1/(p-1)}}{c^{1/(p-1)}+(1-c)^{1/(p-1)}}
\end{equation}
of $\tU_p(c,a):=U_{|\cdot|^p}(c,1,a)$ in $a$. Finally, substitute this minimizer for $a$ in  
$\tR_p(c,a):=\frac{\tU_p(c,0)}{\tU_p(c,a)}$
and simplify, to show that $\hat r_c(p):=\tR_p(c,\ta_{p;c})$ equals the expression under the $\max$ sign in \eqref{eq:tka_p}. 

The continuity of $\tka_p$ in $p$ follows because $\hat r_c(p)$ is 
continuous in $p\in(1,2]$ uniformly in $c\in[0,\frac12]$ \big(indeed, the derivative, $\hat r'_c(p)$, of $\hat r_c(p)$ in $p$ is bounded over all $c\in[0,\frac12]$ and all $p$ in any compact subinterval of $(1,2]$\big). 
That $\tka_2=1$ is trivial. 
To check that $\tka_{1+}=2$, observe that $\tR_p(p-1,p)\to2$ as $p\downarrow1$ and recall that $\ka_f\le2$ for all $f\in\F_{1,2}\setminus\{0\}$. 
The statements that the values of $\tka_p$ are algebraic for all rational $p\in(1,2]$ and $\tka_{3/2}=\frac19\,\sqrt{51 + 21\sqrt7}=1.14\dots$, corresponding to $c=\frac16\,(3 -\sqrt{1 + 2\sqrt7})=0.081\dots$, are straightforward to check. 

It remains to prove that
$\tka_p$ strictly decreases in $p\in(1,2]$. To accomplish this, it is enough to show that $\hat r_c(p)$ does so for each $c\in(0,\frac12)$, since $\hat r_0(p)=\hat r_{1/2}(p)=1$ for all $p\in(1,2]$ and $\hat r_c(2)=1$ for all $c\in[0,\frac12]$. 
Take indeed any $p\in(1,2)$ and $c\in(0,\frac12)$ and 
observe that $(\ln \hat r_c)'(p)=r_1+r_2-\frac1{p-1}r_3$, where 
\begin{align*}
r_1&:=\frac{c^{p-1}\ln c + (1-c)^{p-1}\ln(1-c)}{c^{p-1} + (1-c)^{p-1}}, \\
r_2&:=\ln\big(c^{1/(p-1)}+(1-c)^{1/(p-1)}\big), \\
r_3&:=\frac{c^{1/(p-1)}\ln c+(1-c)^{1/(p-1)}\ln(1-c)}{c^{1/(p-1)}+(1-c)^{1/(p-1)}}. 
\end{align*}
Note that 
$r_1+r_2-\frac1{p-1}r_3=R_1+R_2$, where $R_1:=r_1-r_3$ and $R_2:=r_2 + (1 - \frac1{p - 1}) r_3$. 
Observe that 
\begin{equation}
	R_1=\frac{\left(\left(\frac{1-c}{c}\right)^{p-1}-\left(\frac{1-c}{c}\right)
   ^{\frac{1}{p-1}}\right) c^{p-1+\frac{1}{p-1}} \ln\frac{1-c}{c}}
   {\left(c^{\frac{1}{p-1}}+(1-c)^{\frac{1}{
   p-1}}\right) \left(c^{p-1}+(1-c)^{p-1}\right)}<0,
\end{equation}
since $\frac{1-c}{c} > 1$ and $p - 1 < 1 < \frac{1}{p-1}$. 

It remains to show that $R_2<0$. 
Consider the new variable 
$$b:=\frac{c^{1/(p-1)}}{c^{1/(p-1)}+(1-c)^{1/(p-1)}},$$ 
so that 
$b\in(0,\frac12)$ and 
$c=\frac{b^{p-1}}{b^{p-1}+(1-b)^{p-1}}$. 
Then one can check that 
\begin{equation}
	R_2=h(b):=(p-2) \big(b\ln b+(1-b)\ln(1-b)\big)-\ln\left(b^{p-1}+(1-b)^{p-1}\right) 
\end{equation}
and 
\begin{equation}
	h''(b)b^{2 - p} (1 - b)^{2 - p} (b^{p - 1} + (1 - b)^{p - 1} )^2   
=h_{21}(b)h_{22}(b), 
\end{equation}
where 
\begin{align*}
	h_{21}(b)&:=\big(\tfrac{2-p}{b}-1\big) \big(\tfrac{b}{1-b}\big)^{2-p}+1, \quad
	h_{22}(b):=\big(\tfrac{b}{1-b}\big)^{p-1} \big(\tfrac{p-1}{b}-1\big)-1,  
\end{align*}
with 
$h_{21}'(b)=(p-2) (p-1) \big(\frac{b}{1-b}\big)^{-p}\,(1-b)^{-3}<0$ and 
$h_{22}'(b)=(p-2) \break
\times(p-1) \big(\frac{b}{1-b}\big)^p\,b^{-3}<0$, 
so that both $h_{21}(b)$ and $h_{22}(b)$ are decreasing in $b$. 
Since $h_{21}(\frac12)=2(2-p)>0$, it follows that $h_{21}>0$ on $(0,\frac12)$. 
So, $h''(b)$ equals $h_{22}(b)$ in sign. 
Since $h_{22}(0+)=\infty>0$ and $h_{22}(\frac12)=2(p-2)<0$, both $h_{22}(b)$ and $h''(b)$ switch from $+$ to $-$ as $b$ increases from $0$ to $\frac12$. 
Therefore, $h(b)$ switches from convexity to concavity in $b\in(0,\frac12)$. 
At that, $h(0+)=h(\frac12)=h'(\frac12)=0$. It follows that $h<0$ and hence $R_2<0$. 
This completes the proof of part (II) and thus 
that of the entire Corollary~\ref{cor:concentr}. 
\end{proof}

\subsection{Proof of Theorem~\ref{th:}}\label{proof of th} 


\textbf{(I, II)}\quad By induction and conditioning, parts (I) and (II) of Theorem~\ref{th:} follow immediately from 

\begin{lemma}\label{lem:x,Y}\ Take any $f\in\F_{1,2}\setminus\{0\}$. 
\begin{enumerate}[(I)]
	\item 
For any 
$x\in\R$ and zero-mean r.v.\ $Y$ 
\begin{equation*}
	\E f(x+Y)\le f(x)+C_f\E f(Y).  	
\end{equation*}
	\item
If a constant factor $C$ is such that 
\begin{equation}\label{eq:X,Y,C}
	\E f(X+Y)\le\E f(X)+C\E f(Y)   	
\end{equation}
for all independent zero-mean r.v.'s $X$ and $Y$, then $C\ge C_f$. 		
\end{enumerate}
\end{lemma} 

We shall turn to the proof of this lemma in a moment, after the proof of parts (III) and (IV) of Theorem~\ref{th:} is completed. 

\textbf{(III)}\quad Take any $f\in\F_{1,2}\setminus\{0\}$. 
The inequality $C_f\ge1$ follows by \eqref{eq:C_f}, since $L_{f;s}(x)\to f(s)$ as $x\downarrow0$.  
On the other hand, in view of Proposition~\ref{prop:C_psi} and \eqref{eq:C_f}, one has $L_{\psi_t;s}(x)\le2\psi_t(s)$ for  
any $t\in(0,\infty]$ and $x,s$ such that $0<x<s<\infty$; so, \eqref{eq:f} implies $L_{f;s}(x)\le2f(s)$, whence, by \eqref{eq:C_f}, $C_f\le2$. 

\textbf{(IV)}\quad Part (IV) of Theorem~\ref{th:} follows immediately from Propositions~\ref{prop:C_psi} and \ref{prop:C_powers}. 

Thus, Theorem~\ref{th:} is proved, modulo Lemma~\ref{lem:x,Y}. 

\begin{proof}[Proof of Lemma~\ref{lem:x,Y}]\ 
The main idea of this proof is to use appropriate Taylor expansions. A similar approach was used e.g.\ in \cite{dharm-fab-jog68,chatterji,pin80,pin94,T2,pin12-2smooth}. 

\textbf{(I)}\quad 
Clearly, for all real $z$ and $y$, 
\begin{equation}\label{eq:<hat C_f}
f(z+y)\le f(z)+yf'(z)+\hat C_f f(y),  
\end{equation}
where
\begin{equation}\label{eq:hat C_f:=}
\hat C_f:=\sup_{\substack{z\in\R,\\y\in\R\setminus\{0\}}}R_f(z,y)	\quad\text{and}\quad R_f(z,y):=\frac{f(z+y)-f(z)-yf'(z)}{f(y)}.	
\end{equation}
Here one may recall that, as was noted at the end of the paragraph containing \eqref{eq:f''}, 
$f>0$ on $\R\setminus\{0\}$. Concerning the validity of \eqref{eq:<hat C_f} when $y=0$, recall that $f(0)=0$ and assume that $\hat C_f f(y)=0$ if $y=0$ and $\hat C_f=\infty$ (in fact, later it will be seen that $\hat C_f$ is always between $1$ and $2$. 

It is not hard to see that 
\begin{equation}\label{hat C_f=C_f}
	\hat C_f=C_f. 
\end{equation}
Indeed, because $f$ is an even function and hence $f'$ is an odd function, it follows that $R_f(-z,-y)=R_f(z,y)$ for any $z\in\R$ and $y\in\R\setminus\{0\}$. 
So, one may replace the condition $y\in\R\setminus\{0\}$ in \eqref{eq:hat C_f:=} by $y\in(-\infty,0)$. 
Take indeed any such $y$ and consider the Taylor expansion  
\begin{equation}\label{eq:R f =}
	R_f(z,y)f(y)=f(z+y)-f(z)-yf'(z)=y^2\int_0^1(1-t)f''(z+ty)\dd t. 
\end{equation}
By \eqref{eq:f''}, $f''$ is nondecreasing on the interval $(-\infty,0)$. 
Next, note that $z+ty<0$ whenever $z\in(-\infty,0]$, $y\in(-\infty,0)$, and $t\in(0,1)$.   
Therefore, in view of \eqref{eq:R f =} and the continuity of $f$ and $f'$, $R_f(z,y)$ is nondecreasing in $z\in(-\infty,0]$. 
Similarly, $R_f(z,y)$ is nonincreasing in $z\in[-y,\infty)$, because  $f''$ is nonincreasing on the interval $(0,\infty)$ and $z+ty>0$ whenever $z\in[-y,\infty)$, $y\in(-\infty,0)$, and $t\in(0,1)$.  
Hence, the condition $z\in\R,y\in\R\setminus\{0\}$ in \eqref{eq:hat C_f:=} can be replaced by $y\in(-\infty,0),z\in(0,-y)$. 
Thus, \eqref{hat C_f=C_f} follows by replacing $s$ and $x$ in \eqref{eq:C_f} by $-y$ and $z$, respectively. 

Now part~(I) of Lemma~\ref{lem:x,Y} follows immediately from \eqref{eq:<hat C_f} and \eqref{hat C_f=C_f}.

\textbf{(II)}\quad 
For any positive real numbers $c$ and $d$, let 
$X_{c,d}$ stand for any r.v.\ such that $\P(X_{c,d}=-c)=\frac d{c+d}$ and $\P(X_{c,d}=d)=\frac c{c+d}$. 
Take now any $c$ and $s$ such that 
$0<c<s<\infty$ 
and introduce 
\begin{equation}\label{eq:g,J}
	g_{f;c,s}(x):=\E f(x+X_{c,s-c})-f(x)\quad\text{and}\quad 
	J_{f;c,s}(x):=\frac{g_{f;c,s}(x)}{g_{f;c,s}(0)};  
\end{equation}
the latter definition is correct, because $f>0$ on $\R\setminus\{0\}$ and hence $g_{f;c,s}(0)=\E f(X_{c,s-c})>0$. 

In view of the Taylor expansion in \eqref{eq:R f =}, for any $x\in\R$ 
\begin{align}
	sg_{f;c,s}(x)&=cf(x+s-c)+(s-c)f(x-c)-sf(x) \label{eq:g=}\\
	& =(s-c)c\int_0^1(1-t)\big[(s-c)f''\big(x+(s-c)t\big)+cf''(x-ct)\big]\dd t. \label{eq:g=int}
\end{align}
%
Since $f''$ is even on $\R$ and nonnegative and nonincreasing on $(0,\infty)$, the identity \eqref{eq:g=int} implies that $g_{f;c,s}(x)$ converges to a finite limit as $x\to-\infty$, and then so does $J_{f;c,s}(x)$. 
Let now $a$ and $b$ be any positive real numbers. Then 
\begin{equation*}
	\frac{\E f(X_{a,b}+X_{c,s-c})-\E f(X_{a,b})}{\E f(X_{c,s-c})}
	=\frac b{a+b}J_{f;c,s}(-a)+\frac a{a+b}J_{f;c,s}(b)
	\underset{a\to\infty}\longrightarrow J_{f;c,s}(b), 
\end{equation*}
assuming that the r.v.'s $X_{a,b}$ and $X_{c,s-c}$ are independent. 
So, the constant $C$ in \eqref{eq:X,Y,C} cannot be less than $J_{f;c,s}(b)$, for any $c$, $s$, $b$ such that $0<c<s<\infty$ and $0<b<\infty$. 

On the other hand, by l'Hospital's rule, for any $x\in\R$, 
\begin{equation}\label{eq:J->G} 
	J_{f;c,s}(x)\underset{c\uparrow s}\longrightarrow
	\frac{L_{f;s}(x)}{f(s)},     
\end{equation} 
with $L_{f;s}(x)$ as in \eqref{eq:L}
So, in view of 
\eqref{eq:C_f}, 
$C\ge C_f$. 
So, part (II) of Lemma~\ref{lem:x,Y} is proved as well. 
\end{proof} 


Now Theorem~\ref{th:} is completely proved.

\bibliographystyle{abbrv}


\bibliography{C:/Users/Iosif/Dropbox/mtu/bib_files/citations12.13.12}

\end{document}